\documentclass[12pt]{amsart}

\usepackage{amssymb}
\usepackage{a4wide}
\usepackage{url}

\usepackage{graphicx,subcaption}

\newcommand{\cO}{\mathcal{O}}
\newcommand{\cS}{\mathcal{S}}
\newcommand{\Gd}{\Gamma_{d}}
\newcommand{\N}{\mathbb{N}}
\newcommand{\C}{\mathbb{C}}
\newcommand{\R}{\mathbb{R}}
\newcommand{\Q}{\mathbb{Q}}
\newcommand{\Z}{\mathbb{Z}}
\newcommand{\IK}{\mathbb{K}}
\newcommand{\IP}{\mathbb{P}}
\newcommand{\FPG}[2]{\left\langle\ #1\ |\ #2\ \right\rangle} 
\newcommand{\inn}[2]{\left\langle #1,#2 \right\rangle} 
\newcommand{\br}{\textrm{br}} 
\newcommand{\hc}{\mathbf{H}_{\mathbb{C}}^2}

\newtheorem{thm}{Theorem}[section]
\newtheorem{lem}{Lemma}[section]
\newtheorem{dfn}{Definition}[section]
\newtheorem{prop}[thm]{Proposition}
\theoremstyle{remark}
\newtheorem{rmk}[thm]{Remark}
\def\cqfd{\mbox{}\nolinebreak\hfill$\Box$\medbreak\par}
\newenvironment{pf}{\noindent\textbf{Proof:}}{\cqfd}

\author{Martin Deraux and Mengmeng Xu}

\address{Martin Deraux : Universit\'e Grenoble Alpes, Institut
  Fourier, 100 rue des Math\'ematiques 38610 Gi\`eres; Sorbonne
  Université and Université de Paris, CNRS, INRIA, IMJ-PRG, Ouragan,
  F-75005 Paris, France } \email{martin.deraux@univ-grenoble-alpes.fr}

\address{Mengmeng Xu : School of Mathematics, Hunan University,
  Changsha, China }
\email{mengm\_xu@163.com}

\date{Oct 18, 2022}

\title{Torsion in 1-cusped Picard modular groups}

\begin{document}

\maketitle

\begin{abstract}
  We present a systematic effective method to construct coarse
  fundamental domains for the action of the Picard modular groups
  $PU(2,1,\cO_d)$ where $\cO_d$ has class number one,
  i.e. $d=1,2,3,7,11,19,43,67,163$. The computations can be performed
  quickly up to the value $d=19$.  As an application of this method,
  we classify conjugacy classes of torsion elements, deduce short
  presentations for the groups, and construct neat subgroups of small
 index.
\end{abstract}

\section{Introduction}

The first goal of this paper is to study conjugacy classes of torsion
elements in some Picard modular groups, which we write as
$\Gd=PU(2,1,\cO_d)$, where $\cO_d$ is the ring of algebraic
integers in $\Q(i\sqrt{d})$, and $d$ is a square-free positive
integer. We will mainly treat the cases where $\cO_d$ is a Euclidean
domain, i.e.  for $d=1,2,3,7,11$; our methods are valid more generally
in cases where $\cO_d$ is a unique factorization domain, which is
equivalent to $\Gd$ having exactly one cusp. There are four value
of $d$ where $\Gd$ has one cusp but $\cO_d$ is not Euclidean,
namely $d=19,43,67$ and $163$; even though our code runs in principle
for these values, the computations tend to be very lengthy, and we
only went through with the case $d=19$.

For $d=1$ and $3$, most of what we do can be found by gathering
several papers in the literature
(see~\cite{falbel-parker},~\cite{falbel-francsics-parker},~\cite{deligne-mostow}
and also~\cite{holzapfel-book}). For $d=1,3,7$, presentations were
obtained by Mark-Paupert~\cite{mark-paupert} using coarse fundamental
domains coming from covering depth estimates. The first author
explained in~\cite{deraux-picard} how to push their method further in
order to study torsion in $\Gamma_7$; at the time of that paper, we
had not written computer code to handle more general values of $d$,
which is now settled for 1-cusped Picard modular groups.

For $d=2$ and $11$, presentations were worked out by
Polletta~\cite{polletta} (without mention of the classification of
conjugacy classes of torsion elements).

Matthew Stover pointed out to us that (conjugacy classes of) torsion
elements for \emph{all} Picard modular groups $\Gd$ were listed
by Feustel (see~\cite{feustel}, and
also~\cite{feustel-holzapfel},~\cite{holzapfel-book}). Feustel's
method is very different from ours, and at the time of his paper no
explicit presentation for these groups was known, so his work cannot
be used directly for studying general torsion-free subgroups. In
particular, without our work, it would not at all be obvious to relate
Feustel's work to the presentations worked out by Mark-Paupert and
Polletta.

As far as we know, no explicit presentation for $\Gd$ has appeared in
the literature for $d>11$, so our results for $d=19$ are entirely
new. For $d=43$ and $67$, we were unable to go through with all the
computations, but we did obtain explicit presentations
(see~\cite{computer-code}). For $d=163$ even the covering depth is
unknown at present.

We will give two applications of our classification of isotropy groups
for the action of $\Gd$ on the complex hyperbolic plane. One is
the determination of short presentations for these groups, using
presentations for isotropy groups (see
sections~\ref{sec:torsion-gens}
and~\ref{sec:braid-presentations}).

As a second application, we explain how to use group-theory software
(GAP or Magma) to find explicit neat subgroups of $\Gd$ of small
index, see section~\ref{sec:neat-subgroups}. Recall that a neat
lattice in $PU(2,1)$ is a torsion-free subgroup whose cusps groups
(maximal parabolic subgroups) can be realized by unipotent groups. The
``torsion-free'' requirement is equivalent to the fact that group acts
without fixed points on the complex 2-ball (i.e. the quotient is a
smooth complex hyperbolic surface). The second requirement is
equivalent to the existence of a smooth compactification of the
quotient by elliptic curves (see~\cite{amrt} for arithmetic lattices
and~\cite{mok-projective} for the general case).

Of course every lattice contains many neat subgroups. Indeed, by a
classical result of Selberg (see~\cite{selberg} or~\cite{alperin}), one
can take subgroups obtained as the congruence kernel modulo a suitable
prime ideal (these are called principal congruence
subgroups). Note however that torsion-free/neat congruence subgroups
in $\Gd$ tend to have fairly large index; we would like to find
neat subgroups of smallest possible index.

The basic method we use in order to obtain subgroups of ``small''
index is to start with a given neat \emph{normal} subgroup
$K\subset \Gd$, and to try and enlarge it by replacing it by
$\phi^{-1}(S)$ for some non-trivial subgroup $S\subset F=\Gd/K$
(see section~\ref{sec:neat-subgroups} for more details). In order to
get the initial normal subgroup, we use either computational
group-theory software (Magma), or congruence subgroups.

A basic lower bound for the index of a torsion-free subgroup is
deduced from the fact that the index of a torsion-free subgroup must
be a multiple of the least common multiple of the orders finite
subgroups (see Proposition~(2.1) in~\cite{edmonds-ewing-kulkarni} for
instance). We will refer to this as \emph{the obvious lower bound}.

Note that for subgroups of $PSL_2(\R)=PU(1,1)$, the obvious lower
bound is essentially (i.e. up to a factor of 2) the minimal index of a
torsion-free subgroup, see~\cite{edmonds-ewing-kulkarni}. For
$PSL_2(\C)={\textrm Isom}(H^3_{\R})$ however, there are lattices where
the ratio between the minimal index and the obvious lower bound is
aribitrarily large (see~\cite{jones-reid}). In general very little is
known about the minimal index of torsion-free subgroups of lattices.

The obvious lower bound is actually realized for $d=1$ and $3$ (at
least for $d=3$, this is well-known to experts,
see~\cite{parker-cusps},~\cite{stover-cusps}). For other
values of $d$, we could only find subgroups of index strictly larger
than the obvious lower bound (see Table~\ref{tab:neat-record}). In
fact for $d=2$ and $7$, the index we found is twice the obvious lower
bound, whereas for $d=11$ and $19$, the index of the neat subgroups we
found is quite a bit larger than the obvious lower bound (note however
that in these last cases, $\Gd$ seems to contain too many
subgroups for our methods to be efficient).\\

\noindent
\textbf{Acknowledgements:} The first author would like to thank
Matthew Stover for useful discussions around the results in this
paper, and for pointing our attention to the results by Feustel and
Holzapfel. He also acknowledges support from INRIA, in the form of a
research semester in the ``Ouragan'' team, and thanks Fabrice
Rouillier and Owen Rouill\'e for useful discussions related to this
project. The second author is grateful for the support of China
Scholarship Council (CSC Grant No. 202006130069) and the encouragement
from Yueping Jiang. She also grateful for the hospitality of the
Institut Fourier, where most of this work was done. Finally, the
authors thank the anonymous referees, whose suggestions helped improving
earlier versions of the manuscript. \\

\noindent
\textbf{Conflict of interest:} The authors declare that they have no
conflict of interest in publishing this work.

\section{Background and notation}

\subsection{The complex hyperbolic plane}

In this section we review basics of complex hyperbolic geometry, using
notation close to~\cite{mark-paupert},~\cite{deraux-picard}. We refer
to~\cite{goldman-book} for more detail.

On the complex vector space $V=\C^3$, we define the Hermitian form
$\inn{v}{w}=w^*Jv$, where
$$
J = \left(
  \begin{matrix}
    0 & 0 & 1\\
    0 & 1 & 0\\
    1 & 0 & 0
  \end{matrix}\right).
$$
Note that this Hermitian form has signature $(2,1)$, so the unitary
group $U(J)=\{A\in GL(V):A^*JA=J\}$ is isomorphic to $U(2,1)$.

We write $p:V\setminus\{0\}\rightarrow \IP(V)$ for projectivization,
i.e. $p(v)=\C v$ is the complex line spanned by $v$, and write
$V_-=\{v\in V: \inn{v}{v}<0 \}$,
$V_0=\{v\in V:
\inn{v}{v}=0\}$. Vectors in $V_-$ (resp.
$V_0$) are called negative (resp. isotropic).

As a set the complex hyperbolic plane $\hc$ is given by $p(V_-)$,
which clearly admits an action of $PU(J)$ (namely the one induced by
the action of $GL(V)$ on $V$, which preserves $V_-$).

There is a unique (up to scaling) invariant K\"ahler metric on $\hc$,
and it has constant negative holomorphic sectional curvature. In this
paper, we will not need the expression of that K\"ahler metric, but we
will use the formula for the corresponding Riemannian distance
function. When the holomorphic sectional curvature is normalized to be
$-1$, the Riemannian distance $\rho(x,t)$ between $x=p(v)$ and
$y=p(w)$ is given by
$$
\cosh\left(\frac{1}{2}\rho(x,y)\right)=\frac{|\inn{v}{w}|}{\sqrt{\inn{v}{v}\inn{w}{w}}}.
$$

The set $p(V_0)$ is usually called the boundary of the complex
hyperbolic plane, and we denote it by $\partial \hc$. The points of
$\partial \hc$ are called ideal points.

For every $v=(v_1,v_2,v_3)\in V$ such that $v_3\neq 0$, the complex
line $\C v$ is spanned by a unique vector of the form $(z_1,z_2,1)$,
namely $(v_1/v_3,v_2/v_3,1)$; the pair $(z_1,z_2)$ then gives affine
coordinates such that the complex hyperbolic plane is described as the
subset of $(z_1,z_2)\in\C^2$ satisfying
$$
2\Re(z_1)+|z_2|^2<0,
$$
a region which is known as the Siegel half space.

When $v_3=0$, the only vectors $v=(v_1,v_2,0)$ that are in
$V_-\cup V_0$ are proportional to $q_\infty=(1,0,0)$, and it is
natural to call $q_\infty$ the "point at infinity" for the above
affine coordinates. In what follows, with a slight abuse of notation,
we will write $q_\infty$ instead of $p(q_\infty)$.

Another important set of coordinates are horospherical coordinates,
obtained by studying the stabilizer of $(1,0,0)$ in $U(J)$. It is easy
to check that unipotent upper triangular matrices preserve $J$ if and
only if they are of the form
\begin{equation}
T(z,t)=\left(
  \begin{matrix}
    1 & -\bar z & \frac{-|z|^2+it}{2}\\
    0 &   1     &   z\\
    0 &   0     &   1
  \end{matrix}
\right)\label{eq:heisenberg-translation}
\end{equation}
for some $z\in\C$, $t\in\R$. Moreover, these matrices form a subgroup
of $U(J)$, in fact we have
$T(z,t)T(z',t')=T(z+z',t+t'+2\Im(z\bar z'))$.

The corresponding group law on $\C\times \R$
$$
(z,t)\star(z',t')=(z+z',t+t'+2\Im(z\bar z'))
$$
is often call the Heisenberg group law. The matrices $T(z,t)$ are
called \emph{Heisenberg translations}. Note that the center of the
Heisenberg group is given by $\{0\}\times\R$, and these are sometimes
called \emph{vertical translations}.

The above group of unipotent matrices acts simply transitively on
$\partial \hc\setminus \{q_\infty\}$. This suggests using
$(z,t)\in\C\times \R$ as coordinates on
$\partial \hc\setminus \{q_\infty\}$. These in fact extend to
coordinates on $\hc$, by writing
$(z_1,z_2)=(\frac{-|z|^2+it-u}{2},z)$, where $z\in\C$, $t,u\in
\R$. Note that the point $(\frac{-|z|^2+it-u}{2},z)$ is in $\hc$
(resp. $\partial \hc$) if and only if $u>0$ (resp. $u=0$). The
parameter $u$ is called "horospherical height", and the level sets
$u=u_0$ (resp. the sup-level sets $u\geq u_0$) are called horospheres
(resp. horoballs) based at $q_\infty$.

The full parabolic stabilizer of
$q_\infty=(1,0,0)$ is larger than the above unipotent
subgroup, it is generated by the unipotent stabilizer and the subgroup
of Heisenberg rotations, which are given by the diagonal matrices
$$
R_w=\left(
  \begin{matrix}
    1 & 0 & 0\\
    0 & w & 0\\
    0 & 0 & 1
  \end{matrix}
\right)
$$
where $w\in\C$, $|w|=1$. The matrices of the form $R_wT(z,t)$ with
$w\neq 1$ are called \emph{twist parabolic} elements.

Heisenberg translations and rotations preserve the $Cygan$ $metric$, which
is defined for $(z,t)$ and $(z',t')\in\mathbb{C}\times\mathbb{R}$ by:
\begin{equation}\label{eq:cygan-distance}
\begin{aligned}
d_{C}((z,t),(z',t'))&=\big\lvert\lvert z-z'\rvert^{4}+\lvert t-t'+2\Im(z\overline{z'})\rvert^{2}\big\rvert^{1/4}\\
&=\lvert 2\langle\psi(z,t,0),\psi(z',t',0)\rangle\rvert^{1/2},
\end{aligned}
\end{equation}
where 
$$
\psi(z,t,u)=
\left(
  \begin{matrix} (-|z|^2+it-u)/2 \\
    z \\
    1
  \end{matrix}
  \right).
$$
The Cygan metric is the restriction to $\mathbb{C}\times\mathbb{R}$ of
the \emph{extended Cygan metric}, which is defined for
$(z,t,u)$ and $(z',t',u')\in\mathbb{C}\times\mathbb{R}\times\mathbb{R}_{\geq0}$
by
\begin{equation}
  \nonumber
  d_{Cy}((z,t,u),(z',t',u'))=\big\lvert(\lvert z-z'\rvert^{2}+\lvert u-u'\rvert)^{2}+\lvert t-t'+2\Im(z\overline{z'})\rvert^{2}\big\rvert^{1/4}.
\end{equation}
When at least one of $u$, $u'$ is 0, we get:
\begin{equation}
\nonumber
d_{Cy}((z,t,u),(z',t',u'))=\lvert 2\langle\psi(z,t,u),\psi(z',t',u')\rangle\rvert^{1/2}.
\end{equation}
%\subsection{Cygan spheres}
Given $A\in U(2,1)$, we define the isometric sphere of $A$ to be given
by
$$
I(A)=\{p(V)\in\hc:\lvert\langle V,q_{\infty}
\rangle\rvert=\lvert\langle V,A(q_{\infty})\rangle\rvert\}.
$$

\begin{dfn}\label{dfn:ford}
  Let $\Gamma\subset PU(J)$ be a discrete subgroup. The Ford domain
  for $\Gamma$ is defined by
  $$
  F_\Gamma = \left\{ p(x)\in\hc: |\inn{x}{q_\infty}|\leq |\inn{x}{A q_\infty}| \textrm{ for all } A\in \Gamma \right\}.
  $$
\end{dfn}
It is a standard fact that $F_\Gamma$ is a fundamental domain for the
action of $\Gamma$ modulo the action of the stabilizer
$Stab_{\Gamma}(q_\infty)$, in the sense that its images under the
group tile $\hc$ and $\gamma F_\Gamma\cap F_\Gamma$ has non-empty
interior if and only if $\gamma$ fixes $q_\infty$
(see~\cite{beardon} for instance). In order to get an actual
fundamental domain, we need to intersect $F_\Gamma$ with a fundamental
domain for the action of $Stab_{\Gamma}(q_\infty)$.

Note that it can be delicate to determine the combinatorics of
$F_\Gamma$ explicitly, and in fact the point of the methods developed
in~\cite{mark-paupert} is to avoid working out the combinatorial
structure of the Ford domain.

It turns out (see Proposition 4.3 of~\cite{kim-parker}) that the
isometric sphere of a group element is actually a sphere for the Cygan
metric, whose radius and center can be obtained from a matrix
representative, as stated in Lemma~\ref{lem:center-radius}.
\begin{lem}\label{lem:center-radius}
  Let $A\in U(2,1)$, and suppose $q_\infty$ is not fixed by $A$.
  Then $I(A)$ is equal to the extended Cygan sphere $\cS_{A}$ with
  center $A(q_\infty)$ and radius $\sqrt{2/\lvert A_{3,1}\rvert}$.
\end{lem}
From this, it also follows that the Ford domain $F_\Gamma$ can also be
thought of as the intersection of the exteriors of the Cygan spheres
$\cS_A$ for all elements $A\in \Gamma$ not fixing the point at
infinity.

\subsection{Picard modular groups}

In this section, we let $d>0$ be a square-free integer, and write
$\IK_d=\Q(i\sqrt{d})$, $\cO_d$ for the ring of algebraic integers in
$\IK_d$. Recall that $\cO_d=\{a+b\tau_d:a,b\in\Z\}$, where
$\tau_d=\frac{1+i\sqrt{d}}{2}$ if $d\equiv 3$ mod 4, and
$\tau_d=i\sqrt{d}$ otherwise.

Recall that for most values of $d$, the only units in $\cO_d$ are
$\pm 1$; the only exceptions are the cases $d=1$ (where the units are
the 4-th roots of unity) and $d=3$ (where units are the 6-th roots of
unity).

We now consider $U(J,\cO_d)=U(J)\cap GL_3(\cO_d)$, and write
$\Gd=PU(J,\cO_d)$. These are often called Picard modular
groups.

It follows from a very general result of Borel-Harish
Chandra~\cite{bhc} that $\Gd$ is a lattice in $PU(J)$,
i.e. the quotient $\Gd\backslash \hc$ has finite volume. It
follows from thick-thin decomposition that the quotient has finitely
many ends, the ends corresponding to conjugacy classes of maximal
parabolic subgroups in $\Gd$. Moreover these maximal
parabolic subgroups are given by the stabilizers in $\Gd$ of
$\IK_d$-rational points (see~\cite{borel}), i.e. vectors in $\IP(V)$
that can be represented by a vector in $\cO_d^3$.

The following result is well known.
\begin{thm} \label{thm:feustel-zink}
  (Feustel~\cite{feustel-cusps},Zink~\cite{zink}) The number of ends
  of the quotient $\Gd\backslash \hc$ is given by the class
  number of $\IK_d$.
\end{thm}

From now on, we always assume that the class number of $\IK_d$ is one,
which is equivalent to requiring that $\cO_d$ is a unique
factorization domain. There are finitely many values of $d$ such that
this happens, namely $d=1,2,3,7,11,19,43,67,163$ (note that $\cO_d$ is
in fact a Euclidean ring if and only if $=1,2,3,7,11$).

We briefly review some terminology from~\cite{mark-paupert} (see
also~\cite{deraux-picard}). A vector $v\in\cO_d^3$ is called primitive
if for every $0\neq \alpha\in\cO_d$, $\frac{1}{\alpha}v\in\cO_d$
implies that $\alpha$ is a unit. Since we assume $\cO_d$ is a unique
factorization domain, this is equivalent to requiring that the
greatest common divisor of the standard coordinates $v_1,v_2,v_3$ is
1. Moreover, every $\IK_d$-rational point in $\IP(V)$ has a primitive
representative, and that representative is unique up to multiplication
by a unit in $\cO_d$.

This ensures that the following definition is meaningful.
\begin{dfn}
  The \emph{depth} of an $\cO_d$-rational point $x$ is given by
  $|\inn{v}{q_\infty}|^2=|v_3|^2$, where $v$ is any primitive integral
  lift of $x$.
\end{dfn}
In fact the possible depths of rational points are precisely rational
integers that are norms in $\cO_d$.

By extension, we will also talk about the depth of an element
$A\in U(J)$, defined to be the depth of $A(q_\infty)$. Note that
$A(q_\infty)$, which is the first column of $A$, is a primitive
integral vector (this can easily be seen from the fact that
$A^*JA=J$).

Note in particular that, by Lemma~\ref{lem:center-radius}, the depth
$N$ of an element $A\in U(J)$ is closely related to the radius of the
Cygan sphere $\cS_A$, which is given by $(4/N)^{1/4}$.

The following result is easy to see from
equation~\ref{eq:cygan-distance}, we will use it throughout the paper
(see also Lemma~4 in~\cite{mark-paupert}).
\begin{prop}
  Let $A\in U(J)$ be an element of depth $N$. Then the maximum
  horospherical height of the Cygan sphere $\cS_A$ is given by
  $2/\sqrt{N}$.
\end{prop}
%Recall from the following result of~\cite{mark-paupert}.
%\begin{prop} Let $S_{A}$ be the extended cygan
%  sphere defined as in Lemma~\ref{lem:center-radius}, and $H_{u_{0}}$
%  be the horosphere based at $\infty$ at height $u_{0}>0$. Then
%  $H_{u_{0}}\cap A(H_{u_{0}})=H_{u_{0}}\cap S_{A}$.
%\end{prop}
%In our program, we use the following corollary many times.

\subsection{Cusps of Picard modular groups} \label{sec:cusps-picard}

We write $\Gd^{(\infty)}$ for the stabilizer of $q_\infty$ in
$\Gd$. Explicit generators for $\Gd^{(\infty)}$ as well
as a fundamental domain for its action in $\C\times\R$ (hence on any
horosphere based at $q_\infty$) can be found in~\cite{paupert-will}
(see also~\cite{mark-paupert}), we simply state their result without
proof.

Note that the cases $d=1$ and $d=3$ are special because of the
presence of non-trivial units, and the case $d=2$ is quite different
from the cases $d\geq 7$ because of the congruence class of $d$ mod
$4$.

The fundamental domain can be chosen to be a prism
$P = T\times[0,2\sqrt{d}]$ with base a triangle $T$ with vertices
$0,\lambda,\mu$ where $\lambda\in\R_+$. The values of $\lambda$, $\mu$
depend on $d$ in the following manner.
\begin{table}[htbp]
  \centering
  \begin{tabular}{c||c|c}
    $d$ & $\lambda$ & $\mu$\\
    \hline
    1   &   1       &  $1+i$\\
    2   &   2       &  $i\sqrt{2}$\\
    3   &   1       &  $\frac{1+\tau_3}{3}$\\
  d=7,11,19,43,67,163  &   1       &  $\tau_d$
  \end{tabular}
  \caption{The base of the fundamental prism for the action is the
    triangle with vertices $0,\lambda,\mu$.}\label{tab:triangle}
\end{table}

Explicit generating sets for $\Gd^{(\infty)}$ are described
in~\cite{paupert-will}. The vertical generator is always given by
$T_v=T(0,2\sqrt{d})$.  We list non-vertical generators in
Table~\ref{tab:cusp-generators}, each generator being given in the
form $T(z,t)$ for $(z,t)\in\C\times\R$ (see
equation~\eqref{eq:heisenberg-translation}).
\begin{table}[htbp]
  \centering
  \begin{tabular}{c||c|c|c|c}
    $d$                 & $r$ & $T_r$              & $\mu$       & $T_\mu$\\
\hline
    $1$                 & 2   & $T(2,0)$           & $1+i$       & $T(1+i,0)$\\
    $2$                 & 2   & $T(2,0)$           & $i\sqrt{2}$ & $T(i\sqrt{2},0)$\\
    $3$                 & 1   & $T(1,i\sqrt{3})$   & $\tau_3$    & $T(\tau_3,i\sqrt{3})$\\    
    $7$                 & 1   & $T(1,i\sqrt{7})$   & $\tau_7$    & $T(\tau_7,0)$\\
    $d=11,19,43,67,163$ & 1   & $T(1,i\sqrt{d})$    & $\tau_d$    & $T(\tau_d,i\sqrt{d})$    
  \end{tabular}
  \caption{Non-vertical generators for
    $\Gd^{(\infty)}$.}\label{tab:cusp-generators}
\end{table}
Note that these generating sets \emph{do not} give side-pairing maps
for
$P$, hence we briefly explain how to bring a given point
$(z,t)\in\C\times\R$ back to the fundamental prism
$P$. In what follows, we refer to the $\C$ (resp.
$\R$) factor as the horizontal (resp. vertical) factor.

The rough idea is to adjust the horizontal factor by using powers of
$T_r$ and $T_\mu$, then adjusting the vertical factor by using powers
of $T_v$. We now briefly explain the details of this procedure.

In order handle the horizontal factor, we write the parallelogram
spanned by $r$ and $\mu$ as a union of explicit images of the
triangular base $T$ of $P$, as illustrated in
Figure~\ref{fig:decoupe} for various values of $d$.
\begin{figure}[htbp]
  \centering
  \hfill
  \begin{subfigure}{0.4\textwidth}
    \includegraphics[height=3cm]{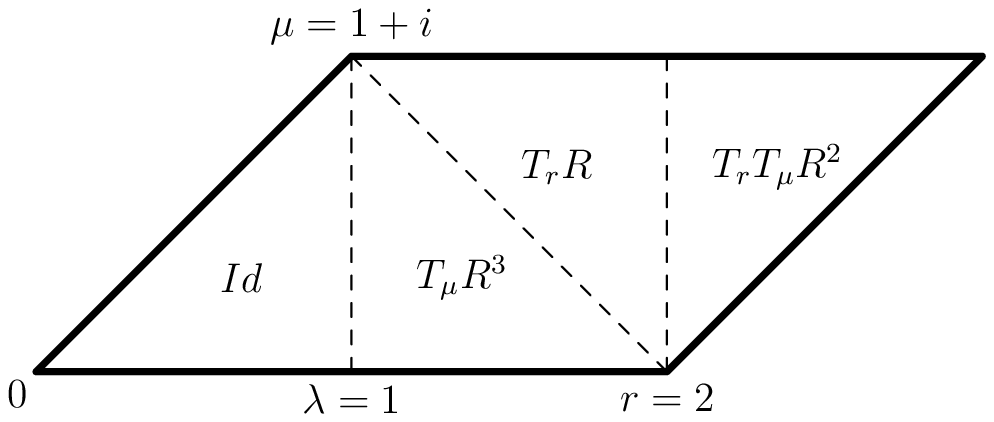}
    \caption{$d=1$}
  \end{subfigure}\hfill
  \begin{subfigure}{0.4\textwidth}
    \includegraphics[height=3cm]{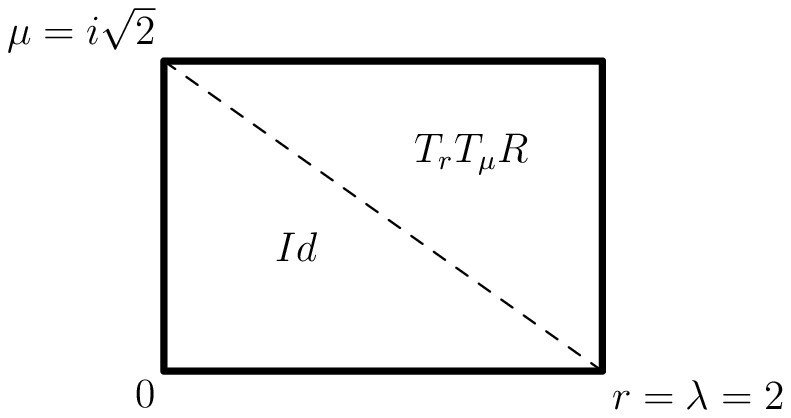}
    \caption{$d=2$}
  \end{subfigure}\hfill
  \\[1cm]
  \hfill
  \begin{subfigure}{0.4\textwidth}
    \includegraphics[height=4cm]{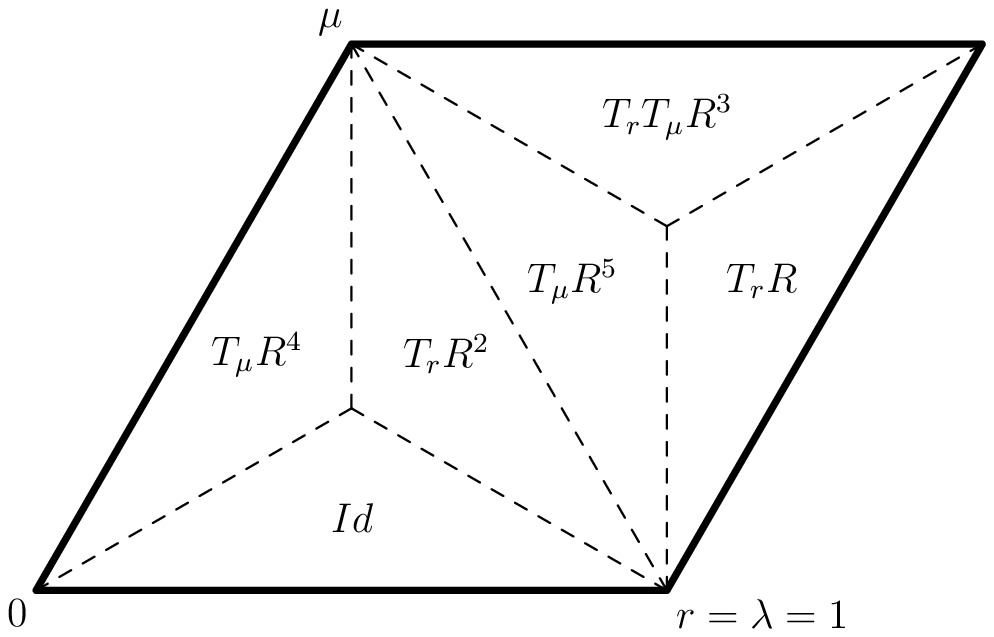}
    \caption{$d=3$}
  \end{subfigure}\hfill
  \begin{subfigure}{0.4\textwidth}
    \includegraphics[height=4cm]{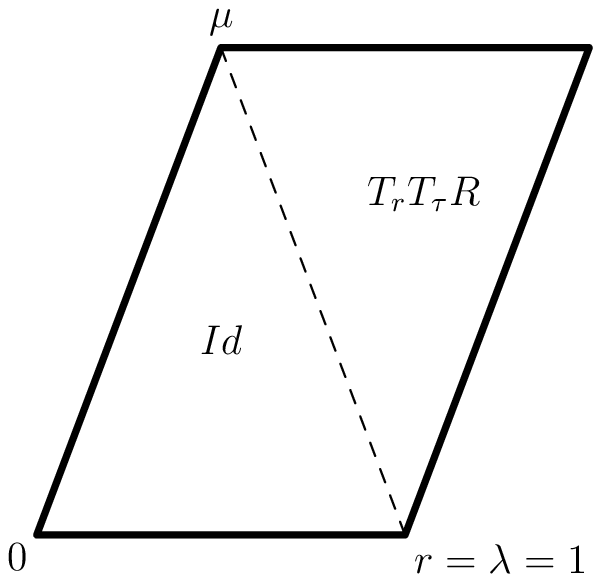}
    \caption{$d\geq 7$}
  \end{subfigure}\hfill
  \caption{Cutting the basic parallelogram as a union of the triangle $T$, which
is the base of the prism $P$.}\label{fig:decoupe}
\end{figure}
Write $z=r \alpha+\beta\mu$ for some $\alpha,\beta\in\R$, and compute
$a=\lfloor \alpha\rfloor$ and $b=\lfloor \beta\rfloor$ (note that we
will only perform such floor/ceiling calculations when $\alpha$,
$\beta$ are real algebraic numbers, so the result can be certified
with a computer).

By replacing $(z,t)$ by the Heisenberg coordinates of
$T_\mu^{-b}T_r^{-a}V(z,t)$, we may assume that $z$ is in the parallelogram
spanned by $r$ and $\mu$. Applying a suitable element of
$\Gd^{(\infty)}$ if necessary (see Figure~\ref{fig:decoupe}), we may
assume that $z$ is in the triangle with vertices
$0,\lambda,\mu$. Finally, applying a suitable power of $T_v$, we can
ensure that $t$ is in the interval $[0,2\sqrt{d}]$.

Note that using the above method, we find one element of
$\Gd^{(\infty)}$ that brings our point to $P$, but in general
there can be several such elements. In fact, if the corresponding
point is in the interior of $P$, then the element is unique. If the
image point is in the boundary of $P$, even though there can be
several choices of elements of $\Gd^{(\infty)}$, it is
straightforward to write a computer program that lists all
possibilities.

Using the above, it is also easy to write code to check whether two
points are equivalent under the action of the standard cusp group
$\Gd^{(\infty)}$, and if so, to list all elements of
$\Gd^{(\infty)}$ that map one to the other. Note that this works
for ideal points, but also for pairs of points with the same
horospherical height.

Finally, we mention that with the above method, we are able to list,
for each depth $N$, a single representative of every
$\Gd^{(\infty)}$-orbit of rational points of depth $N$ (we will
choose a representative whose Heisenberg coordinates are inside the
prism $P$).

\section{Estimates for cusp elements} \label{sec:bounds}

As mentioned in~\cite{deraux-picard}, in order to make our methods
effective, we will need a priori bounds on the rational points that
are useful to perform the computations.  We will use two basic bounds.

One is the list of $\IK_d$-rational points of depth $\leq N$ such that
the corresponding Cygan sphere intersects the prism $P$ at a given
horospherical height $u_0$.

The other bound will be associated to pairs of Cygan spheres $\cS_1$,
$\cS_2$ associated two rational points $p_1$, $p_2$ in the prism
$P$. Given two such points, we will need a bound on the cusp elements
$\alpha$ such that $\alpha(\cS_1)$ intersects $\cS_2$.

In both cases, rather than finding the precise list of cusp elements
satisfying a property, we will find an explicit finite set that
contains all the cusp elements safisfying that property. Even though
the precise list could in principle be deduced from that upper bound,
this is not efficient in practice, because it would take too much
computation time.

We now sketch one way to get such bounds. We only explain the first
bound, the second one is similar; in fact, since we our methods only
require to consider images of Cygan spheres $\alpha(\cS_1)$ that
intersect both $S_2$ \emph{and the prism $P$}, we can use the first
bound, and then use a simple triangle inequality estimate for the
second (if $\alpha(\cS_1)\cap S_2\neq\emptyset$, then
$d_{Cy}(\alpha(p_1),p_2)<r_1+r_2$, where $r_j$ is the Cygan radius of
$S_j$).

For computational purposes, it is convenient to use slightly modified
Heisenberg coordinates, in order to get the Cygan spheres bounding the
Ford domain to have equations given by polynomials with
$\Z$-coefficients.

Accordingly, we scale the vertical Heisenberg coordinate by $\sqrt{d}$
and use $\tilde{t}=t/\sqrt{d}$. The corresponding scaled Heisenberg
coordinates $(z,\tilde{t})$ of $v=(v_0,v_1,v_2)$ then satisfy
\begin{equation}\label{eq:heisenberg}
\begin{array}{c}
  \frac{v_0}{v_2} = \frac{-|z|^2+i\tilde{t}\sqrt{d}}{2}\\
  \frac{v_1}{v_2} = z
\end{array}.
\end{equation}
%We perform this change of coordinate mainly so that the interested
%reader can compare the formulas with the ones used in our computer
%code, see~\cite{computer-code}.

In the following, we denote by $\cS_V$ the Cygan sphere corresponding
to a given a primitive integral vector $V=(v_0,v_1,v_2)\in\cO_d$.
We wish to find restrictions on $V$ under the hypotheses that $\cS_V$
intersects the cone over $P$ from $q_\infty$ at a given
horospherical height $u_0$. More precisely, let $u_0>0$ and consider
the vertical translation $P_{u_0}$ of $P$ at horospherical height $u$,
i.e. the set of points in $\hc$ with horospherical coordinates
$(z,t,u_0)$ such that $(z,t)\in P$. Then we have the following.

\begin{prop} \label{prop:bound} Let $N\in\N^*$, and let
  $u_0\in\R$,
  $u_0>0$. Then there are only finitely many primitive vectors
  $V$ of depth $N$ such that $\cS_V\cap P_{u_0}\neq
  \emptyset$. In fact, writing
  $V=(v_0,v_1,v_2)$ for
  $v_j=a_j+b_j\tau_d$, the integers
  $a_j,b_j\in\Z$ must satisfy the bounds in
  equations~\eqref{eq:bound-b1},~\eqref{eq:bound-a1},~\eqref{eq:bound-b0}.
\end{prop}

\begin{proof}
  The fact that $|v_2|^2=N$ gives a small list of possible values for
  $v_2\in\cO_d$ (note that there are efficient algorithms for listing
  the numbers $z\in\cO_d$ such that $|z|^2=N$, even for large $N$; our
  computer code~\cite{computer-code} uses the PARI command
  \verb|bnfintnorm|).
  
  In what follows, we fix one such value of $v_2=a_2+b_2\tau_d$, and
  show how to find restrictions on $v_1$; then, given $v_2$ and $v_1$,
  we explain how to find restrictions on $v_0$.

  The general stream of arguments to show that there are
  finitely many possibilites for $v_j=a_j+b_j\tau_d$, consists in
  first bounding $b_j$, then bounding $a_j$ in terms of $b_j$ (the
  corresponding bounds will also depend on $v_k$ for $k<j$).

  Denote by $c_0$ the center of the circumscribed circle of the base $T$
  of the prism $P$, which has vertices $0,\lambda,\mu$ (see
  Table~\ref{tab:triangle}). Note that
  $$
  c_0 = \frac{\lambda\mu(\bar\lambda-\bar\mu)}{\bar\lambda\mu - \lambda\bar\mu},
  $$
  and by definition any $z\in T$ satisfies $|z-c_0|\leq |c_0|$.
  
  We write $(z_0,\tilde{t}_0)$ for the (scaled) Heisenberg coordinates
  of the center of $\cS_V$; recall that this Cygan sphere has radius
  $(4/N)^{1/4}$. In particular, the equation of the corresponding
  Cygan ball can be written as
  \begin{equation}\label{eq:cygan-sphere}
    (|z-z_0|^2+u)^2 + (\sqrt{d}(\tilde{t}-\tilde{t}_0) + 2\Im(z\bar z_0) )^2 \leq \frac{4}{N},
  \end{equation}
  with equality corresponding to the Cygan sphere.

  If there is a point with horospherical coordinates $(z,\tilde{t},u)$
  satisfying equation~\eqref{eq:cygan-sphere}, then
  $$
  |z-z_0|\leq \sqrt{\frac{2}{\sqrt{N}}-u}.
  $$
  Moreover, if the point satisfies $(z,t)\in P$, then we have
  $$
  |z_0-c_0|\leq \sqrt{\frac{2}{\sqrt{N}}-u} + |c_0|.
  $$
  Equation~\eqref{eq:heisenberg} then gives
  $$
  |v_1-c_0v_2|\leq R,
  $$
  where $R = \sqrt{N}(\sqrt{\frac{2}{\sqrt{N}}-u} + |c_0|)$.
  
  Recall that $v_1\in\cO_d$, so we can write $v_1=a_1+b_1\tau_d$ for
  some $a_1,b_1\in \Z$. Writing $c_0v_2=\alpha_1+\beta_1\tau_d$ for
  $\alpha_1,\beta_1\in\Q$, we have
  $$
  |a_1-\alpha_1+(b_1-\beta_1)\tau_d|^2 = (a_1-\alpha_1+(b_1-\beta_1)\Re\tau_d)^2 + (b_1-\beta_1)^2 (\Im\tau_d)^2 \leq R^2,
  $$
  in particular
  \begin{equation}\label{eq:bound-b1}
    \lceil \beta_1 - \frac{R}{\Im\tau_d}\rceil \leq b_1\leq \lfloor \beta_1 + \frac{R}{\Im\tau_d}\rfloor,
  \end{equation}
  so we get a bound for $b_1$. Given a $b_1$ that satisfies this bound, we then have
  \begin{equation}\label{eq:bound-a1}
    \lceil \alpha_1-(b_1-\beta_1)\Re\tau_d - R \rceil \leq a_1\leq \lfloor \alpha_1-(b_1-\beta_1)\Re\tau_d + R\rfloor.
  \end{equation}

  In particular there are finitely many possible values for $v_1$. Now
  given $v_2,v_1\in \cO_d$ satisfying these bounds, we explain how to
  bound the possible values of $v_0$. We wish to use the fact that
  $$
  (\sqrt{d}(\tilde{t}-\tilde{t}_0) + 2\Im(z\bar z_0) )^2 \leq \frac{4}{N}-u_0^2,
  $$
  for some $z\in T$.
  
  We write $z_0=x_0 \lambda + y_0\mu$, $z=x \lambda + y \mu$ for some
  $x,y,x_0,y_0\in\R$, and compute
  $$
  \Im(z\bar z_0)=(x_0y-xy_0)\lambda\Im\mu.
  $$
  Note that $z$ is in $T$, which is the convex hull of $0,\lambda,\mu$
  so $0\leq x,y\leq 1$, hence we have
  $$
  |\Im(z\bar z_0)| \leq (|x_0|+|y_0|)\lambda\Im\mu.
  $$
  
  Now we get
  $$
  -\sqrt{\frac{4}{N}-u_0^2} - 2 (|x_0|+|y_0|)\lambda\Im\mu \leq \sqrt{d}(\tilde{t}_0-\tilde{t}) \leq \sqrt{\frac{4}{N}-u_0^2} + 2 (|x_0|+|y_0|)\lambda\Im\mu,
  $$
  and $\tilde{t}\in[0,2]$, so we get a bound
  \begin{equation}\label{eq:t0-bounds}
    \tilde{t}_0^{min}\leq \tilde{t}_0\leq \tilde{t}_0^{max},
  \end{equation}
  where
  $$
  \begin{array}{c}
    \tilde{t}_0^{min} = - (\sqrt{\frac{4}{N}-u^2} + 2 (|x_0|+|y_0|)\lambda\Im\mu)/\sqrt{d}\\
    \tilde{t}_0^{max} = 2 + (\sqrt{\frac{4}{N}-u^2} + 2 (|x_0|+|y_0|)\lambda\Im\mu)/\sqrt{d}
  \end{array}
  $$
  Note however that $\tilde{t}_0$ is not an integer, so we need to work a little
  more to get an effective method.
  
  We now use equation~\eqref{eq:heisenberg}, which relates $\tilde{t}_0$ to
  $v_0=a_0+b_0\tau_d$ ($a_0,b_0\in\Z$). Taking the real and imaginary
  parts of both sides of the equation
  $$
  a_0 + b_0\tau_d = v_2\left( \frac{-|z_0|^2+i\tilde{t}_0\sqrt{d}}{2} \right),
  $$
  we get
  \begin{equation}\label{eq:t0-v0}
    \begin{array}{c}
      a_0 + b_0\Re\tau_d = - \frac{1}{2}|z_0|^2\Re v_2 - \frac{1}{2}\tilde{t}_0 \sqrt{d} \Im v_2\\
      b_0\Im\tau_d = - \frac{1}{2}|z_0|^2\Im v_2 + \frac{1}{2}\tilde{t}_0 \sqrt{d}  \Re v_2\\
    \end{array}
  \end{equation}
  If $\Re v_2\neq 0$, the second line of equation~\eqref{eq:t0-v0} can be solved to $\tilde{t}_0$, hence we get
  \begin{equation}\label{eq:bound-b0}
    \tilde{t}_0^{min} \leq \frac{b_0\Im\tau_d + \frac{1}{2}|z_0|^2\Im
      v_2}{\frac{1}{2}\sqrt{d} \Re v_2} \leq \tilde{t}_0^{max},
  \end{equation}
  which in turn gives us a bound for $b_0$.
  
  Similarly, the first equation~\eqref{eq:t0-v0} can be solved for $\tilde{t}_0$
  (at least when $\Im v_2\neq 0$), to get a bound on $a_0$.
  
  The cases when $\Re v_2$ and/or $\Im v_2$ are 0 are actually easier,
  because equation~\eqref{eq:t0-v0} gives more restrictive conditions on
  $a_0$ and $b_0$.
\end{proof}

\begin{rmk}\label{rmk:crude-bounds}
  \begin{enumerate}
  \item The bounds given in the proof of Proposition~\ref{prop:bound}
    are by no means optimal, but they allow us to run a certified
    computer search in a finite amount of time. For future reference,
    we refer to the bounds obtained in the proof as \emph{crude
      bounds} (they are the ones used in our computer code,
    see~\cite{computer-code}).
  \item One can shorten the list of Cygan spheres that satisfy the
    crude bounds, by checking whether each sphere in the list actually
    intersects the fundamental prism at horospherical height
    $u_0$. This can be done by using elementary calculus, and
    certifying the results by using some computational tool like the
    Rational Univariate Representation (RUR), see~\cite{rouillier}. In
    our computer program, we run an indermediate refinement of the
    bound using the RUR as implemented in giac (see~\cite{giac}), and
    restrict to Cygan spheres whose projection to all three coordinate
    axes intersect the projection of the fundamental prism. This
    allows us to speed up the computation for large values of $d$.
  \end{enumerate}
\end{rmk}

\section{Effective Feustel-Zink} \label{sec:feustel-zink}

We observe that the Feustel-Zink result
(Theorem~\ref{thm:feustel-zink}) can be rephrased as follows (recall
that $q_\infty=(1,0,0)$, and we assume throughout the paper that
$\IK_d$ has class number one).
\begin{prop}\label{prop:feustel-zink}
  For every primitive integral vector $v\in \cO_d^3$, there exists
  $A\in \Gd$ such that $A(q_\infty)=v$.
\end{prop}
As mentioned in~\cite{mark-paupert}, it is not obvious how to make
this statement effective, we will sketch how our computer code does
this.

The first remark is that it is enough to find a matrix $A$ as in
Proposition~\ref{prop:feustel-zink} only for $v$ in a list of
representatives for $\Gd^{(\infty)}$-orbits of rational points (we
sketched in section~\ref{sec:cusps-picard} how such a list can be
gathered).

The second observation is that rather than the mere existence of $A$
as in the statement of Proposition~\ref{prop:feustel-zink}, we may be
more restrictive and assume that $A^{-1}(q_\infty)$ is also in our
list of representatives for $\Gd^{(\infty)}$-orbits of rational
points. Indeed, for any $M\in\Gd^{(\infty)}$,
$AM^{-1}(q_\infty)=Aq_\infty$ (so we can replace $A$ by $AM^{-1}$),
and $(AM^{-1})^{-1}(q_\infty)=MA^{-1}(q_\infty)$.

Now recall that the inverse of an element $A\in U(J)$ is given by
$A^{-1}=JA^*J$, so the first column of $A$ and the last row of
$A^{-1}$ are obtained from one another by complex conjugation and
multiplication by $J$ (the last one amounts to flipping the first and
third entry of the vector). Note in particular that $A$ and $A^{-1}$
have the same depth.

This means that when searching for $A\in U(J)$ whose first column is
$v$, we may assume the last row of $A$ is given up to a multiplication
by a unit in $\cO_d$ by $(Jw)^*$ for some $w$ in the list of
representatives for $\Gd^{(\infty)}$-orbits of rational points of
depth given by $|v_2|^2$.

We then try to fill in the upper right 2$\times$2 matrix (see
Proposition~\ref{prop:fill}); if this fails, we try another of the
finitely many possibilities for the last row of $A$. This method turns
out to be very efficient in practice, it allowed us to construct all
necessary matrices in a fairly short amount of computation time for
$d\leq 19$.

We are grateful for an anomymous referee for having communicated to us
the result of Proposition~\ref{prop:fill}, which simplifies some
clumsy computations that appeared in earlier versions of the
manuscript.
\begin{prop}\label{prop:fill}
  Let $A\in U(J)$ be written as
  $$ A=\left(
  \begin{matrix}
    a & x & y\\
    b & z & w\\
    c & d & e
  \end{matrix}
  \right).
  $$
  Assume $c\neq 0$ and let $\delta=\det(A)$. Then we have
  $x=\frac{1}{\delta c}(\delta ad+\bar b)$,
  $z=\frac{1}{\delta c}(\delta bd-\bar c)$,
  $w=\frac{1}{\delta c}(\delta be+\bar d)$, and
  $y=(ae\bar c-\bar \delta \bar b\bar d+c)/|c|^2$.  This matrix is in
  $\Gd$ if and only if its entries are in $\cO_d$, and in that case
  $\delta=\det(A)$ is a unit in $\cO_d$.
\end{prop}
\begin{pf}
  The fact that $A^*JA=J$ can be rewritten as
  $A^{-1}=JA^*J$; the proposition then follows from easily by
  expressing the entries of $A^{-1}$ in terms of cofactors of $A$.
\end{pf}

\section{Covering depths}

The key to our computations is to obtain a bound on the radius of the
Cygan spheres that intersect the Ford domain for $\Gd$ (see
Definition~\ref{dfn:ford}). More specifically, let us write
$1=n_1<n_2<n_3<\dots$ for the depths of Cygan spheres for elements of
$\Gd$ (this is equivalent to listing the norms in $\cO_d$ in
increasing fashion).

Following~\cite{mark-paupert}, we define the \emph{covering depth} of
$\Gd$ to be the smallest $n_k$ such that the spheres of depth $n_j$
for $j>k$ do not intersect the Ford domain (note that this does not
necessarily mean that some spheres of depth $n_k$ are necessary to
define the Ford domain). If this is the case, then the Cygan spheres
corresponding to rational points of depth $n_k$ for $k>j$ do not
intersect the Ford domain, and in particular it is enough to use the
rational points of depth $\leq n_j$ in order to study/describe the
Ford domain (in most cases, the Ford domain can actually be described
with an even smaller set of depths, but we will not use this).
%
%Note that this is not quite the definition used in~\cite{mark-paupert}
%and~\cite{polletta}, but it is closely related to it. We briefly
%summarize an algorithm for computing this covering depth (our
%algorithm is inspired by the method used by Mark-Paupert and
%Polletta).

The basic procedure gives us a way to answer the following :

\begin{figure}[htbp]
  \textsc{Problem} :
  For a given $j\in\N^*$, determine whether or not the cross section
  at horospherical height $u=2/\sqrt{n_{j+1}}$ of the fundamental
  domain for the standard cusp group $\Gd^{(\infty)}$ is covered by the
  interiors of the Cygan spheres of depth $\leq n_j$.
\end{figure}

The covering depth is then the smallest $n_j$ such that the answer to
our \textsc{Problem} is YES.

In order to do this, we will
\begin{enumerate}
\item reduce the verification to finitely many checks (note that the Ford
  domain has infinitely many sides);
\item subdivide the prism into convex pieces that are small enough for
  each piece to be contained in a Cygan sphere centered
  at a rational point of depth $\leq n_j$.
\end{enumerate}

The reduction to finitely many verifications~(1) requires an explicit
(finite) upper bound on the set of rational points $v\in\cO_d^3$ of
depth $\leq n_j$ such that the Cygan sphere $S_v$ corresponding to $v$
(see Lemma~\ref{lem:center-radius}) satisfies
$S_v\cap C_P\neq\emptyset$. In fact, we can be a bit more restrictive
and require that $S_v$ intersects $C_P$ at horospherical height
$2/\sqrt{n_{j+1}}$.
%In order to do this, we will produce for each Cygan sphere $S_k$
%centered at a rational point of depth $\leq n_j$ in the prism $P$ a
%finite upper bound $T_k$ for the set
%$$
%\left\{\alpha\in\Gamma_d^{(\infty)}\ :\ \alpha(S_k)\cap C_P\neq \emptyset\right\}.
%$$
The details of how this can be done were explained in
section~\ref{sec:bounds}.

Part~(2) was performed in~\cite{mark-paupert} and~\cite{polletta} by a
search ``by hand'' of a decomposition inspired by visual analysis of
pictures of Cygan spheres. We will give a more systematic method,
based on a dichotomy method in the prism. It is probably far from
optimal (and runs quite slowly for large values of $d$), but it has
the advantage that it does not rely on human intervention/visual
inspection.

The procedure will maintain a list of prisms that need to be studied
(until we get an answer to our question), initialized as containing
just one prism, namely the fundamental prism $L_0 = \{P\}$.

We now explain how to construct the list $L_{k+1}$ from the list $L_k$
(or stop the procedure if we have reached an answer).
%Let $Q$ be an
%element of $L_k$ (think of $P$, or a smaller prism obtained by
%subdividing $P$).

For any $u>0$, and for any prism $Q$ in the Heisenberg group, we write
$Q_u$ for the translate of $Q$ at horospherical height $u$. Now for
every prism $Q$ in $L_k$, we do the following:
\begin{itemize}
\item For each vertex $v$ of the translate $Q_{2/\sqrt{n_{j+1}}}$ of
  $Q$, list the Cygan spheres of depth $\leq n_j$ that contain $v$
  in their interior. If this list is empty, we know the prism is NOT
  covered (answer reached).
\item If every vertex of $Q_{2/\sqrt{n_{j+1}}}$ is covered by some
  Cygan sphere, check if there is a single Cygan sphere that contains
  all of its vertices.
  \begin{itemize}
  \item If so, the prism $Q_{2/\sqrt{n_{j+1}}}$ is covered by a Cygan
    sphere, and we do not include it in $L_{k+1}$.
  \item If not, subdivide $Q$ into $2^3=8$ smaller prisms, and include
    them in $L_{k+1}$.
  \end{itemize}  
\end{itemize}

If the answer to our \textsc{PROBLEM} is NO, then there exists a prism
obtained from the above dichotomy with at least one vertex not covered
by any Cygan sphere, so the procedure will stop and find that the answer
is NO.

If the answer is YES, once again, this will be seen at the level of
some fine enough decomposition of the prism into prisms at scale
$1/2^k$ for some $k$, so after finitely many stages we will get
$L_k=\emptyset$.

\begin{rmk}
  As in~\cite{mark-paupert}, in order to certify the inequalities used
  to verify whether a given vertex of a prism in the subdivision is
  covered by a Cygan sphere, in our computer program, we replace the
  horospherical height $2/\sqrt{n_{j+1}}$ by a rational approximation,
  i.e. a number $u_j\in\Q$ that satisfies
  $2/\sqrt{n_{j+1}}<u_j<2/\sqrt{n_j}$, which we choose to be "close"
  to $2/\sqrt{n_{j+1}}$. This is inconsequential for our purpose,
  which is to determine a finite list of depths that suffice to define
  the Ford domain for the corresponding Picard modular group.
\end{rmk}

Running this procedure, we find the covering depths given in
Table~\ref{tab:covering-depths}. In each case, we give the number of
cusp orbits of rational points of depth at most equal to the covering
depth; in parentheses, we list the number of cusp orbits remaing after
removing the centers of Cygan spheres obviously contained in another
one (using the triangle inequality, i.e. comparing the Cygan distance
with the sum of the radii of the spheres).
\begin{table}[htbp]
  \centering
  \begin{tabular}[htbp]{|c|c|c|c|}
    \hline
    $d$ & Covering depth  & Size of smallest prism used & Number of cusp orbits of rational points\\
    \hline
    1   & 4               &  $1/2^2$                    & 4 (4)\\
    2   & 16              &  $1/2^5$                    & 46 (46)\\
    3   & 4               &  $1/2$                      & 4 (4)\\
    7   & 7               &  $1/2^5$                    & 8 (8)\\
    11  & 36              &  $1/2^{19}$                 & 226 (198)\\
    19  & 64              &  $1/2^7$                    & 540 (455)\\
    43  & 269             &  $1/2^{14}$                 & ? (6184)\\
    67  & 607             &  $1/2^{15}$                 & ? (26098)\\
    163 & $\geq 3053$     &  $\leq 1/2^6$               & ?\\
    \hline
  \end{tabular}
  \caption{Covering depths for 1-cusped Picard modular groups.} 
    \label{tab:covering-depths}
\end{table}

For $d=19$, running all the computations (covering depth,
presentation, classification of isotropy groups) already takes several
hours, and the computation time seems prohibitive for larger values of
$d$ (at least in our implementation).

\section{Computation times}

In table~\ref{tab:cpu-time}, we gather rough computation time for
various values of $d$, and various parts of the computations. For
$d\geq 43$, our Sage implementation~\cite{computer-code} of the method
is inefficient (both in computation time and memory usage), and it
only allowed us to go through with part of the computation (see the
question marks in Tables~\ref{tab:covering-depths}
and~\ref{tab:cpu-time}). We hope that a better implementation will
allow us to treat $d=43$, $67$ and perhaps even $163$.

\begin{table}[htbp]
  \centering
  \begin{tabular}[htbp]{|c||c|c|c|c|c|}
    \hline
    $d$ & Covering depth (giac)      & Matrices (giac)           & Torsion    & Presentation & Conversion\\
    \hline
    1   &      4 s                   &           0 s             &  34 s      & 0.5 s        & 35 s\\
    2   &  1 min 11 s                &           5 s             & 2 min 32 s & 24 s         & 3 min 16 s\\
    3   &     6 s                    &           0 s             & 1 min 4 s  & 0.5 s        & 22 s\\
    7   &  14 s                      &           0 s             &  31 s      & 2 s          & 58 s\\
    11  & 4 min 27 s                 &        2 min 6 s          & 4 min 27 s & 5 min 43 s   & 13 min 16s\\
    19  & 14 min 51 s                &        6 min 0 s          & 9 min 11 s & 16 min 15 s  & 54 min 10 s\\
    43  & 19 h 26 min 44 s           &            ?              &     ?      &      ?       & ?\\
    67  &  8 d 18 h 37 min 33s       &            ?              &     ?      &      ?       & ?\\
    163 &       ?                    &            ?              &     ?      &      ?       & ?\\
    \hline
  \end{tabular}
  \caption{Approximate CPU time on an Intel 1.8GHz processor.}
    \label{tab:cpu-time}
\end{table}

\section{Isotropy groups} \label{sec:isotropy}

In the tables in this section, we describe the non-trivial (conjugacy
classes of) isotropy groups for the action on complex hyperbolic space
of $PU(2,1,\cO_d)$ for $d=1,2,3,7,11$ and $19$. These can also be
thought of as being the non-trivial maximal finite subgroups of
$PU(2,1,\cO_d)$.

The conjugacy classes of complex reflections are listed in
Tables~\ref{tab:refl_1},~\ref{tab:refl_2},~\ref{tab:refl_3},~\ref{tab:refl_7}
for various values of $d$.  Representatives for the conjugacy classes
of isotropy groups with isolated fixed points are listed in
Tables~\ref{tab:nonrefl_1},~\ref{tab:nonrefl_2},~\ref{tab:nonrefl_3},~\ref{tab:nonrefl_7},~\ref{tab:nonrefl_11}
and~\ref{tab:nonrefl_19}.

For an isotropy group $G$ with an isolated fixed point, we write $R_G$
for its complex reflection subgroup, and describe $R_G$ by giving
vectors polar to the mirrors of generators (fifth column), as well as
braid lengths of pairs of generators (fourth column). The order of $G$
(resp. $R_G$) is given in the second (resp. third) column. The number
of mirrors (seventh column) in each group is written as $j_1,j_2,j_3$
where $j_k$ is the number of mirrors in the $\Gamma$-orbit of the
$k$-th polar vector in the list of complex reflections (the latter
vectors are listed in
Tables~\ref{tab:refl_1},~\ref{tab:refl_2},~\ref{tab:refl_3}, etc.)

Almost all finite reflection groups that occur in this way are
well-generated, i.e. can be generated by $2=\dim(\C^2)$
generators. The only exception is one of the isotropy groups for
$\Gamma_2$, which is isomorphic to the
Shepard-Todd~\cite{shephard-todd} group $G_{12}$, see~\cite{brmaro}
for a presentation of that group.

The isotropy groups that are not generated by reflections are all
cyclic; for such groups, we list a (regular elliptic) generator in the
last column of the table.

In the tables below, for any positive integer $k$, we write
$\zeta_{k}=e^{2\pi i/k}$.

%\subsection{The case $d=1$}

\begin{table}[htbp]
  \centering
  \begin{tabular}[htbp]{c|c|c}
    Order & $w$ & $||w||^2$\\
    \hline
    $4$   & $(0,1,0)$   & $1$ \\ 
    $4$   & $(1,-1,0)$  & $1$ \\ 
    $2$   & $(1,0,1)$   & $2$ \\ 
  \end{tabular}    
  \caption{Conjugacy classes of complex reflections for $d=1$.}
  \label{tab:refl_1}
\end{table}

{\tiny
\begin{table}
  \begin{tabular}[htbp]{c|c|c|c|c|c|c|c}
    $v$                             & $|G|$ & $|R_G|$ & br & $w$        & $||w||^2$ & $\#$ mirrors & Extra generators \\
    \hline
    \hline
    $(1+i,-i,-1-i)$                 &  $6$  &  $6$    & 3  & $(0,1+i,1)$& $2$       & $0,0,3$      &    \\
                                    &       &         &    & $(1,0,1)$  & $2$       &              &    \\
    \hline
    $(1,0,-1)$                      &  $8$  &  $8$    & 2  & $(0,1,0)$  & $1$       & $1,0,1$      &    \\
                                    &       &         &    & $(1,0,1)$  & $2$       &              &    \\
    \hline
    $(1,-1,-1)$                     & $32$  &  $32$   & 4  & $(0,1,1)$  & $1$       & $0,2,4$      &    \\
                                    &       &         &    & $(1,0,1)$  & $2$       &              &    \\
    \hline
    \hline
    $(1,i(\zeta_8-1),i(\zeta_8-1))$ & $8$   &  $4$    &    & $(1,-1,0)$ & $1$       & $0,1,0$      & $\left(\begin{matrix} -i & -i-1 & i\\ i-1 & i & 0 \\ i-1 & i-1 & 1\end{matrix}\right)$\\
    \hline
    $(1,0,-\zeta_{12})$             & $12$  &  $4$    &    & $(0,1,0)$  & $1$       & $1,0,0$      & $\left(\begin{matrix} i & 0 & 1\\ 0 & -1 & 0 \\ 1 & 0 & 0\end{matrix}\right)$ \\
  \end{tabular}
  \caption{List of isolated fixed point isotropy groups for
    $d=1$.}\label{tab:nonrefl_1}
\end{table}
}

%\subsection{The case $d=2$}
%
%In Table~\ref{tab:nonrefl_2}, we write $\alpha=i\sqrt{2}$.

\begin{table}[htbp]
  \centering
  \begin{tabular}[htbp]{c|c|c}
    Order & $w$ & $||w||^2$\\
    \hline
    $2$   & $(0,1,0)$   & $1$ \\ 
    $2$   & $(1,-1,0)$  & $1$ \\ 
    $2$   & $(1,0,1)$   & $2$ \\ 
  \end{tabular}    
  \caption{Conjugacy classes of complex reflections for $d=2$.}
  \label{tab:refl_2}
\end{table}

{\tiny
\begin{table}
  \begin{tabular}[htbp]{c|c|c|c|c|c|c|c}
    $v$                                         & $|G|$ & $|R_G|$ & br & $w$                           & $||w||^2$ & $\#$ mirrors & Extra generators \\
    \hline
    \hline
    $(1,0,-1)$                                  &  $4$  &  $4$    &  2 & $(0,1,0)$                     & $1$       & $0,1,1$      &    \\
                                                &       &         &    & $(1,0,1)$                     & $2$       &              &    \\
    \hline
    $(-1-\alpha,\alpha,\alpha)$                 & $4$   &  $4$    &  2 & $(1,-1,0)$                    & $1$       & $0,1,1$      &    \\
                                                &       &         &    & $(1,-\alpha,-\alpha)$         & $2$       &              &    \\
    \hline
    $(-2-\alpha,-1,\alpha)$                     & $6$   &  $6$    &  3 & $(1,-\alpha,0)$               & $2$       & $0,0,3$      &    \\
                                                &       &         &    & $(1+\alpha,0,1)$              & $2$       &              &    \\
    \hline
    $(1,-1,-1)$                                 & $8$   &  $8$    &  4 & $(0,1,1)$                     & $1$       & $0,2,2$      &    \\
                                                &       &         &    & $(1,0,1)$                     & $2$       &              &    \\
    \hline
    $(-2-\alpha,-1+\alpha,\alpha)$              & $48$  &  $48$   &  6 & $(0,\alpha,1)$                & $2$       & $0,0,12$     &    \\
                                                &       &         &    & $(1+\alpha,2,1-\alpha)$       & $2$       &              &    \\
                                                &       &         &    & $(1+\alpha,2-\alpha,-\alpha)$ & $2$       &              &    \\
    \hline
    \hline
%%   $(1,1-\zeta_8,\frac{-1+i(1-\sqrt{2})}{2})$  & $8$   &  $2$    &    & $(\alpha,1,0)$               & $1$       & $1,0,0$      & $\left(\begin{matrix} 1-\alpha & -2 & \alpha\\ -\alpha & -1 & 0 \\ 1 & \alpha & 1\end{matrix}\right)$ \\
    $\notin\IK_d^3$                                & $8$   &  $2$    &    & $(\alpha,1,0)$               & $1$       & $1,0,0$      & $\left(\begin{matrix} 1-\alpha & -2 & \alpha\\ -\alpha & -1 & 0 \\ 1 & \alpha & 1\end{matrix}\right)$ \\
    \hline
%%    $(1,-1-\zeta_8^3,-\frac{1+i}{2})$           & $4$   &  $2$    &    & $(\alpha,1,1)$               & $1$       & $0,1,0$      & $\left(\begin{matrix} 1 & 2 & -2-2\alpha\\ -\alpha & -1-2\alpha & -2+2\alpha \\ -1-\alpha & -2-\alpha & -1+2\alpha\end{matrix}\right)$    
    $\notin\IK_d^3$           & $4$   &  $2$    &    & $(\alpha,1,1)$               & $1$       & $0,1,0$      & $\left(\begin{matrix} 1 & 2 & -2-2\alpha\\ -\alpha & -1-2\alpha & -2+2\alpha \\ -1-\alpha & -2-\alpha & -1+2\alpha\end{matrix}\right)$    
  \end{tabular}
  \caption{List of isolated fixed point isotropy groups for
    $d=2$ (we write $\alpha=i\sqrt{2}$).}\label{tab:nonrefl_2}
\end{table}
}

%\subsection{The case $d=3$}
%
%In the Table~\ref{tab:nonrefl_3}, we write $\omega=e^{2\pi i/3}$.

\begin{table}[htbp]
  \centering
  \begin{tabular}[htbp]{c|c|c}
    Order & $w$ & $||w||^2$\\
    \hline
    $6$   & $(0,1,0)$   & $1$ \\ 
    $2$   & $(1,0,1)$   & $2$ \\ 
  \end{tabular}    
  \caption{Conjugacy classes of complex reflections for $d=3$.}
  \label{tab:refl_3}
\end{table}

{\tiny
\begin{table}
  \begin{tabular}[htbp]{c|c|c|c|c|c|c|c}
    $v$                                     & $|G|$ & $|R_G|$ & br & $w$                          & $||w||^2$ & $\#$ mirrors & Extra generators \\
    \hline
    \hline
    $(1,0,-1)$                              &  $12$ &  $12$   &  2 & $(0,1,0)$                    & $1$       & $1,1$        &    \\
                                            &       &         &    & $(1,0,1)$                    & $2$       &              &    \\
    \hline
    $(1,0,\bar\omega)$                      & $72$  &  $72$   &  4 & $(0,1,0)$                    & $1$       & $2,6$        &    \\
                                            &       &         &    & $(1,-1,-\omega)$             & $2$       &              &    \\
    \hline
    \hline
%%    $(1,-\zeta_9^2+\zeta_9-1, \zeta_9^5)$   & $3$   &  $1$    &    &                              &           &              & $\left(\begin{matrix} \omega & \omega & 1\\ 1 & -\omega & 0 \\ 1 & 0 & 0\end{matrix}\right)$ \\
    $\notin\IK_d^3$                          & $3$   &  $1$    &    &                              &           &              & $\left(\begin{matrix} \omega & \omega & 1\\ 1 & -\omega & 0 \\ 1 & 0 & 0\end{matrix}\right)$ \\
    \hline
%%    $(i-2\bar\omega-2,1,-i+\zeta_{12}+1)$   & $4$ & $2$ &    & $(1,1,-\omega)$              & $2$       & $0,1$        & $\left(\begin{matrix} \omega & -1 & \omega+3\\ 1 & 1 & -i\sqrt{3} \\ 1 & 0 & -i\sqrt{3}\end{matrix}\right)$
    $\notin\IK_d^3$                          & $4$ & $2$ &    & $(1,1,-\omega)$              & $2$       & $0,1$        & $\left(\begin{matrix} \omega & -1 & \omega+3\\ 1 & 1 & -i\sqrt{3} \\ 1 & 0 & -i\sqrt{3}\end{matrix}\right)$
  \end{tabular}
  \caption{List of isolated fixed point isotropy groups for $d=3$ (we
    write $\omega=e^{2\pi i/3}$).}\label{tab:nonrefl_3}
\end{table}
}

%\subsection{The case $d=7$}
%
%In this section (in particular in Table~\ref{tab:nonrefl_7}), we write
%$\tau=\frac{1+i\sqrt{7}}{2}$.

\begin{table}[htbp]
  \centering
  \begin{tabular}[htbp]{c|c|c}
    Order & $w$ & $||w||^2$\\
    \hline
    $2$   & $(0,1,0)$   & $1$ \\ 
    $2$   & $(1,0,1)$   & $2$ \\ 
  \end{tabular}    
  \caption{Conjugacy classes of complex reflections for $d=7,11,19$.}
  \label{tab:refl_7}
\end{table}

{\tiny
\begin{table}
  \begin{tabular}[htbp]{c|c|c|c|c|c|c|c}
    $v$                                     & $|G|$ & $|R_G|$ & br & $w$                          & $||w||^2$ & $\#$ mirrors & Extra generators \\
    \hline
    \hline
    $(1,0,-1)$                              &  $4$  &  $4$    &  2 & $(0,1,0)$                    & $1$       & $1,1$        &    \\
                                            &       &         &    & $(1,0,1)$                    & $2$       &              &    \\
    \hline
    $(3+i\sqrt{7}, 1, \bar\tau)$            & $6$   &  $6$    &  3 & $(1,-\tau,0)$                & $2$       & $0,3$        &    \\
                                            &       &         &    & $(1+i\sqrt{7},0,1)$          & $2$       &              &    \\
    \hline
    $(-1-\tau, -\bar\tau,\tau)$             & $6$   &  $6$    &  3 & $(0,\tau,1)$                 & $2$       & $0,3$        &    \\
                                            &       &         &    & $(\tau,\bar\tau,0)$          & $2$       &              &    \\
    \hline
    $(\bar\tau,0,-1)$                        & $8$   &  $8$    &  4 & $(0,1,0)$                    & $1$       & $2,2$        &    \\
                                            &       &         &    & $(\tau,1,1)$                 & $2$       &              &    \\
    \hline
    $(4, \bar\tau, -\tau)$                 & $8$   &  $8$    &  4 & $(1+i\sqrt{7},0,1)$          & $2$       & $0,4$        &    \\
                                            &       &         &    & $(\tau,\bar\tau,0)$          & $2$       &              &    \\
    \hline\hline
%%    $(\frac{1}{2}((1-\omega)i\sqrt{7}-\omega-1),1,-\omega)$  & $6$ & $2$ &  & $(-\bar\tau,\tau,1)$ & $1$ & $1,0$    & $\left(\begin{matrix} -2& \tau-2 & 1+2i\sqrt{7}\\ \tau-1 & \tau+1 & 3 \\ \tau & 1 & 3-\tau\end{matrix}\right)$ \\
    $\notin\IK_d^3$                          & $6$ & $2$ &  & $(-\bar\tau,\tau,1)$ & $1$ & $1,0$    & $\left(\begin{matrix} -2& \tau-2 & 1+2i\sqrt{7}\\ \tau-1 & \tau+1 & 3 \\ \tau & 1 & 3-\tau\end{matrix}\right)$ \\
    \hline
%%    $(\zeta_{14}^5+2\zeta_{14}^3+\zeta_{14}^2,1,-\zeta_{14}^5+1)$   & $7$  & $1$ &    &     &     &     & $\left(\begin{matrix} -2-\tau & \tau-3 & 3\tau-2\\ -1 & 0 & -\bar\tau \\ -\bar\tau & \tau & 2\end{matrix}\right)$
    $\notin\IK_d^3$                          & $7$  & $1$ &    &     &     &     & $\left(\begin{matrix} -2-\tau & \tau-3 & 3\tau-2\\ -1 & 0 & -\bar\tau \\ -\bar\tau & \tau & 2\end{matrix}\right)$
  \end{tabular}
  \caption{List of isolated fixed point isotropy groups for $d=7$ (we
    write $\tau=\frac{1+i\sqrt{7}}{2}$)}\label{tab:nonrefl_7}
\end{table}
}

%\subsection{The case $d=11$}
%
%In this section (in particular in Table~\ref{tab:nonrefl_11}, we write
%$\tau=\frac{1+i\sqrt{11}}{2}$).
%
%\begin{table}[htbp]
%  \centering
%  \begin{tabular}[htbp]{c|c|c}
%    Order & $w$ & $||w||^2$\\
%    \hline
%    $2$   & $(0,1,0)$   & $1$ \\ 
%    $2$   & $(1,0,1)$   & $2$ \\ 
%  \end{tabular}    
%  \caption{Conjugacy classes of complex reflections for $d=11$.}
%  \label{tab:refl_11}
%\end{table}

{\tiny
\begin{table}
  \begin{tabular}[htbp]{c|c|c|c|c|c|c|c}
    $v$                                     & $|G|$ & $|R_G|$ & br & $w$                          & $||w||^2$ & $\#$ mirrors & Extra generators \\
    \hline
    \hline
    $(2,1-\tau,-2)$                         &  $4$  &  $4$    &  2 & $(1,0,1)$                    & $2$       & $1,1$        &    \\
                                            &       &         &    & $(1+\bar\tau,-4,-1-\bar\tau)$& $2$       &              &    \\
    \hline
    $(\tau+3, 0, \bar\tau)$                 & $4$   &  $4$    &  2 & $(0,1,0)$                    & $1$       & $1,1$        &    \\
                                            &       &         &    & $(\tau-4,0,\tau)$            & $2$       &              &    \\
    \hline
    $(1,0,-1)$                              & $4$   &  $4$    &  2 & $(0,1,0)$                    & $1$       & $1,1$        &    \\
                                            &       &         &    & $(1,0,1)$                    & $2$       &              &    \\
    \hline
    $(3\tau-6, \tau, \tau+2)$              & $6$   &  $6$    &  3 & $(-\bar\tau,\tau,1)$         & $2$       & $0,3$        &    \\
                                            &       &         &    & $(\tau-4,0,\tau)$            & $2$       &              &    \\
    \hline
    $(-\bar\tau,0,1)$                       & $8$   &  $8$    &  4 & $(0,1,0)$                    & $1$       & $2,2$        &    \\
                                            &       &         &    & $(\tau,1,1)$                 & $2$       &              &    \\
    \hline
    $(\tau, 2-\tau,-\tau)$                  & $12$  &  $12$   &  6 & $(1,0,1)$                    & $2$       & $0,6$        &    \\
                                            &       &         &    & $(1,\bar\tau,-\tau)$         & $2$       &              &    \\
    \hline\hline
%%    $(1,\frac{1}{10}(\sqrt{11}(3-i)+3i-11),-1)$ & $4$ & $2$ &  & $(1,0,1)$ & $2$ & $0,1$    & $\left(\begin{matrix} \bar\tau & -\tau-1 & \tau\\ \tau-2 & i\sqrt{11} & -\tau+2 \\ \tau & \tau+1 & \bar\tau\end{matrix}\right)$ \\
    $\notin\IK_d^3$                          & $4$ & $2$ &  & $(1,0,1)$ & $2$ & $0,1$    & $\left(\begin{matrix} \bar\tau & -\tau-1 & \tau\\ \tau-2 & i\sqrt{11} & -\tau+2 \\ \tau & \tau+1 & \bar\tau\end{matrix}\right)$ \\
    \hline
%%    $(1,\frac{-\sqrt{11}-i+4}{2},\frac{-i\sqrt{11}+2i-1}{2})$ & $4$ & $2$ &  & $(\tau,1,1)$ & $2$ & $0,1$    & $\left(\begin{matrix} -1 & \bar\tau & \tau+1 \\ -1 & \tau & 1 \\ \tau & 2 & \bar\tau+1\end{matrix}\right)$ \\
    $\notin\IK_d^3$                          & $4$ & $2$ &  & $(\tau,1,1)$ & $2$ & $0,1$    & $\left(\begin{matrix} -1 & \bar\tau & \tau+1 \\ -1 & \tau & 1 \\ \tau & 2 & \bar\tau+1\end{matrix}\right)$ \\
    \hline
%%    $(1,\frac{(38i+211)i\sqrt{11}+273-56i}{3170},\frac{(41i-398)i\sqrt{11}-86-23i}{3170})$ & $4$ & $2$ &  & $(-\bar\tau,\tau,1)$ & $2$ & $0,1$    & $\left(\begin{matrix} 12 & 6-2\tau & 17\tau-10 \\ 2\bar\tau & -\tau-1 & -\tau-7 \\ 1-3\tau & -\tau-1 & \tau-12 \end{matrix}\right)$ \\
    $\notin\IK_d^3$                          & $4$ & $2$ &  & $(-\bar\tau,\tau,1)$ & $2$ & $0,1$    & $\left(\begin{matrix} 12 & 6-2\tau & 17\tau-10 \\ 2\bar\tau & -\tau-1 & -\tau-7 \\ 1-3\tau & -\tau-1 & \tau-12 \end{matrix}\right)$ \\
  \end{tabular}
  \caption{List of isolated fixed point isotropy groups for
    $d=11$ (we write $\tau=\frac{1+i\sqrt{11}}{2}$).}\label{tab:nonrefl_11}
\end{table}
}

%\subsection{The case $d=19$}
%
%In this section (in particular in Table~\ref{tab:nonrefl_19}, we write
%$\tau=\frac{1+i\sqrt{19}}{2}$.
%
%\begin{table}[htbp]
%  \centering
%  \begin{tabular}[htbp]{c|c|c}
%    Order & $w$ & $||w||^2$\\
%    \hline
%    $2$   & $(0,1,0)$   & $1$ \\ 
%    $2$   & $(1,0,1)$   & $2$ \\ 
%  \end{tabular}    
%  \caption{Conjugacy classes of complex reflections for $d=19$.}
%  \label{tab:refl_19}
%\end{table}

{\tiny
\begin{table}
  \begin{tabular}[htbp]{c|c|c|c|c|c|c|c}
    $v$                                     & $|G|$ & $|R_G|$ & br & $w$                          & $||w||^2$ & $\#$ mirrors & Extra generators \\
    \hline
    \hline
    $(2\tau+4,2,\bar\tau)$                  &  $4$  &  $4$    &  2 & $(2\tau-2,\tau,1)$           & $2$       & $0,2$        &    \\
                                            &       &         &    & $(2\tau,0,1)$                & $2$       &              &    \\
    \hline
    $(1,0,-1)$                              & $4$   &  $4$    &  2 & $(0,1,0)$                    & $1$       & $1,1$        &    \\
                                            &       &         &    & $(1,0,1)$                    & $2$       &              &    \\
    \hline
    $(\bar\tau,0,-2)$                       & $4$   &  $4$    &  2 & $(0,1,0)$                    & $1$       & $1,1$        &    \\
                                            &       &         &    & $(\tau,0,2)$                 & $2$       &              &    \\
    \hline
    $(\tau-6,2\tau,2)$                      & $4$   &  $4$    &  2 & $(-\bar\tau,1,0)$            & $1$       & $1,1$        &    \\
                                            &       &         &    & $(\tau-5,2\tau,2)$           & $2$       &              &    \\
    \hline
    $(\tau-7,-2\bar\tau,\tau+2)$            & $6$   &  $6$    &  3 & $(\tau-3,\tau+1,1)$          & $2$       & $0,3$        &    \\
                                            &       &         &    & $(\tau+2,3,\bar\tau)$        & $2$       &              &    \\
    \hline
    $(5\bar\tau-8,-\tau-3,\tau-3)$          & $6$   &  $6$    &  3 & $(1+\bar\tau,\bar\tau,-1)$   & $2$       & $0,3$        &    \\
                                            &       &         &    & $(\tau+7,2,\bar\tau)$        & $2$       &              &    \\
    \hline
    $(-\bar\tau,0,1)$                       & $8$   &  $8$    &  4 & $(0,1,0)$                    & $1$       & $2,2$        &    \\
                                            &       &         &    & $(\tau,1,1)$                 & $2$       &              &    \\
    \hline\hline
    $(9,1+\bar\tau,-\tau)$                  & $2$   &  $1$    &    &                              &           &              & $\left(\begin{matrix} 9\tau-8 & 9\tau+9 & 81 \\ 2\tau+3 & 8 & 9\bar\tau+9 \\ 5 & 2\bar\tau+3 & 9\bar\tau-8 \end{matrix}\right)$ \\
    \hline
    $\notin\IK_d^3$                          & $4$   & $2$     &    & $(\tau-2,\tau,1)$            & $2$       & $0,1$        & $\left(\begin{matrix} 7 & 3\bar\tau+1 & -9\tau+1\\ \bar\tau+1 & -\tau-2 & -\tau-6 \\ -\tau & -2 & \tau-6 \end{matrix}\right)$ \\
    \hline
    $\notin\IK_d^3$                          & $4$   & $2$     &    & $(\tau-3,\tau+1,1)$          & $2$       & $0,1$        & $\left(\begin{matrix} 2\tau+10 & 6\bar\tau+8 & -9\tau-1\\ 2\bar\tau+6 & -5\tau-1 & -2\tau-11 \\ -3\tau+1 & -\tau-8 & 3\tau-10 \end{matrix}\right)$ \\
    \hline
    $\notin\IK_d^3$                          & $4$   & $2$     &    & $(1+\bar\tau,\bar\tau,-1)$   & $2$       & $0,1$        & $\left(\begin{matrix} 2\tau+25 & 13\bar\tau+10 & -47\tau+20 \\ 2\bar\tau+7 & -5\tau & -12\tau-16 \\ -3\tau+1 & -\tau-7 & 3\tau-26 \end{matrix}\right)$ \\
    \hline
    $\notin\IK_d^3$                          & $6$   & $2$     &    & $(2,-\tau,-1)$               & $1$       & $1,0$        & $\left(\begin{matrix} \tau+1 & \bar\tau+4 & -1-2\tau\\ \bar\tau+2 & -2\tau-2 & \tau-4 \\ -\tau & -3 & -\bar\tau \end{matrix}\right)$ \\
  \end{tabular}
  \caption{List of isolated fixed point isotropy groups for $d=19$ (we
    write $\tau=\frac{1+i\sqrt{19}}{2}$).}\label{tab:nonrefl_19}
\end{table}
}

\section{Neat subgroups} \label{sec:neat-subgroups}

The general method we use to find torsion-free subgroups in
$G=\Gd$ is the following.  We assume we are given a list $T$ of
non-trivial isotropy groups that contains all isotropy groups up to
conjugation in $G$, and generators for the standard cusp $G_\infty$.
\begin{itemize}
\item Find a normal subgroup $K$, and consider $\varphi:G\rightarrow F=G/K$;
\item Check that $\varphi|_t$ is injective for every $t$ in $T$ (if so, $K$ is torsion-free);
\item List subgroups $S$ of $F$ that intersect all conjugates of
  subgroups in $T$ trivially; then $\varphi^{-1}(S)$ is torsion-free and
  $[G:\varphi^{-1}(S)]=[F:S]$.
\item For each such subgroup $S$, study the action of $S$ the right
  cosets of $F_\infty=\varphi(G_\infty)$ in $F$ to find generators for each
  cusp.
\end{itemize}
In the third item, we consider only maximal subgroups with this
property.

About the first item, note that there is an effective algorithm for
listing normal subgroups $H\subset \Gamma$ with $[\Gamma:H]\leq N$ for
any $N\in\N$. This algorithm is implemented in Magma (via the command
\verb|LowIndexNormalSubgroups|), and runs efficiently when $N$ is
``not too large''.

When $G$ has too many (normal) subgroups, reasonable values of $N$
tend to be very small. This seems to be the case for $d=11$ and $d=19$
for instance, where listing all normal subgroups of index $\leq 200$
already takes quite a long time (and none of the many corresponding
subgroups turn out to be torsion-free).  In such cases, we use
congruence subgroups, which tend to produce normal subgroups of larger
index.

The orbifold Euler characteristics of the quotients
$\Gd\backslash \hc$ are known, see Theorem~5A.4.7
in~\cite{holzapfel-book} for instance. For convenience, we list the
results in Table~\ref{tab:neat-record}, as well as the least common
multiple $L_d$ of the orders of finite subgroups in $\Gd$ (it is
a standard fact that the index of any torsion-free subgroup in
$\Gd$ must be a multiple of $L_d$).

We also list the smallest index of torsion-free (TF) subgroups we
could find, as well as the smallest index of a neat subgroup
(i.e. torsion-free and torsion-free at infinity, a property which we
refer to as $TF\infty$). Recall that the last condition is equivalent
to the existence of a compactification of the quotient by finitely
many elliptic curves with negative self-intersection.
\begin{table}[htbp]
  \centering
  \begin{tabular}[htbp]{c|cccc}
    $d$ & $\chi_{orb}$ & $L_d$ & Known $TF\infty$ & Known $TF$\\
    \hline
    1   &  1/32        & 96    &      96          &    96\\
    2   &  3/16        & 48    &      96          &    48\\
    3   &  1/72        & 72    &      72          &    72\\
    7   &  1/7         & 168   &      336         &    336\\
    11  &  3/8         & 24    &      432         &    432\\
    19  &  11/8        & 24    &      864         &    864\\
    43  &  83/8        & ?     &      ?           &     ? \\
    67  &  251/8       & ?     &      ?           &     ? \\
    163 &  2315/8      & ?     &      ?           &     ?
  \end{tabular}
  \caption{Euler characteristics for 1-cusped Picard modular surfaces,
    and least common multiple $L_d$ of the order of finite subgroups.}
\label{tab:neat-record}
\end{table}

In the next few subsections, we will give a bit more details about the
$TF\infty$ subgroups we have found, namely we list their
abelianization, the index of their normal core, and the
self-intersections of the elliptic curves that compactify
them. Computer code to verify our claims is available
at~\cite{computer-code}.

Note that we \emph{do not} claim that the lists in the next few
sections are optimal, nor exhaustive. The reason why we cannot claim
optimality is that we have no reasonable effective upper bound for the
index in $\Gd$ of the normal core
$$
\textrm{Core}_{\Gd}(H) = \cap_{g\in\Gd}gHg^{-1}$$ of a
torsion-free/neat subgroup $H\subset \Gd$ in terms of the index
$n=[\Gd:H]$. It is easy to see that the index of the normal core
is bounded by $n!$, but in practice there is no hope to list all
normal subgroups of index $n!$ (see the obvious lower bound given in
table~\ref{tab:neat-record}, which implies $n\geq 24$).

\subsection{General method}\label{sec:general-method}

In order to produce the lists in sections~\ref{sec:subgroups-1}
through~\ref{sec:subgroups-19}, we ask Magma for a list normal
subgroups of $\Gd$ of index $\leq N$ (we choose $N$ so that Magma
answers within a reasonable amount of time, the size of $N$
depends a lot on the value of $d$, and to a lesser extent on the
presentation used for $\Gd$).

For each normal subgroup $K\subset \Gd$, we determine whether $K$ is
$TF$ (torsion-free); this is done by verifying that the quotient map
$\varphi:\Gd\rightarrow F=\Gd/K$ preserves the order of torsion
elements (it is enough to check that this is the case for a
representative of each conjugacy classes of torsion elements).

We also verify whether $K$ is $TF\infty$ (torsion-free at infinity) by
finding a presentation for its cusps. Note that the cusps of $K$ are
in 1-1 correspondence with right cosets in $F$ of
$F_\infty=\varphi(\Gd^{(\infty)})$, where $\Gd^{(\infty)}$ is the
standard cusp of $\Gd$, i.e. the stabilizer of $(1,0,0)$. In
particular the number of cusps is the index $[F:F_\infty]$. Note also
that since $K$ is a normal subgroup, all its cusps are isomorphic to
each other, and one representative is $K\cap \Gd^{(\infty)}$.

The group $K_\infty=K\cap \Gd^{(\infty)}$ can actually be presented by
computing the kernel of the restricted morphism
$\phi|_{\Gd^{(\infty)}}$, since we have an explicit presentation
for $\Gd^{(\infty)}$.
Given a generating set for $K_\infty$, it is easy to check whether
$K_\infty$ contains twist-parabolic elements, namely we simply check
if every generator is unipotent.

Now for each neat normal subgroup $K\triangleleft \Gd$, we once
again consider the quotient map
$\varphi:\Gd\rightarrow F=\Gd/K$, and search for subgroups
$S\subset F$ such that $\varphi^{-1}(S)$ is $TF$; this amounts to
saying that no non-trivial element of $S$ is conjugate to $\varphi(t)$
for any $t$ in our list of representatives for torsion
elements. Alternatively, this is equivalent to requiring that
$xSx^{-1}\cap X=\{e\}$ for all $x\in F$ and $X=\varphi(\tilde{X})$,
with $\tilde{X}$ any isotropy group in $\Gd$.

If $H=\varphi^{-1}(S)$ is $TF$, we compute generators for each of its
cusps, by studying the action of $S$ on the set of right cosets of
$F_\infty=\varphi(\Gd^{(\infty)})$ in $F$. Note once again that
the cusps are in bijection with $S$-orbits of such right cosets, and
we can produce generators for each cusp by finding generators for the
kernel $K_\infty$ of the map
$\Gd^{(\infty)}\rightarrow F_\infty$, and adjoining extra
generators obtained by lifting generators of the stabilizer in $S$ of
the corresponding right coset.

Once we have generators for the cusps of $H$, we can easily check
whether or not each cusp contains twist-parabolic elements (once
again, simply check whether every element in our generating set is
unipotent).

If no cusp contains twist-parabolics, $H$ is $TF\infty$, and we get
the self-intersection of the compactifying elliptic curves by
computing the abelianization of the cusp groups (this is done by using
a presentation for the cusp groups). Recall that the self-intersection
of the elliptic curve compactifying a given cusp with unipotent group
$U$ is given by the $-k$ where $k$ is the unique positive integer such
that $U\cong \Z_k\oplus\Z\oplus\Z$, see Proposition 4.2.12 and
equation~(4.2.15) of~\cite{holzapfel-book}.

\subsubsection{Slight improvements}

Since we are mainly interested in studying the smallest index of
$TF\infty$ subgroups, for a given finite quotient $F$ of $\Gd$,
we will not study $\phi^{-1}(S)$ for all subgroups $S$ of $F$.

First, we known the index of torsion-free subgroups must be a multiple
of the least common multiple of the order of isotropy groups, which
decreases the list slightly. Also,
\begin{itemize}
\item if $S_1,S_2\subset F$ are
  conjugate, they clearly give conjugate preimages in $\Gd$;
\item if $S_1\subsetneq S_2\subset F$, then
  $\phi^{-1}(S_1)\subset\phi^{-1}(S_2)$ has index $[S_1:S_2]$, and in
  particular in that case $\phi^{-1}(S_1)$ will definitely not be
  optimal.
\end{itemize}
Hence, in searching for subgroups $S$ of $F$, we will discard all
sugroups having one conjugate contained in a subgroup $S'\subset F$
that has already been studied.

\subsubsection{Non-optimality, non-exhaustivity}

As far as we know, no efficient bound is known for the index
$[\Gd:Core(\Gd,H)]$ for subgroups $H\subset \Gd$ of
index $k$ (for example take $k$ to be the smallest index of a
$TF\infty$ subgroup in $\Gd$). If we had such a bound, then we
could in principle run an exhaustive computer search, and determine
the actual optimal index for $TF\infty$ subgroups (this would of
course succeed only if the computer search goes through in reasonable
amount of time and memory).

Once again, we insist that the list of subgroups given here is not
exhaustive.

\subsection{Cusp data for some $TF\infty$ subgroups of index 96 in $\Gamma_1$} \label{sec:subgroups-1}

In order to get Table~\ref{tab:subgroups-1}, we used normal subgroups
in $\Gamma_1$ of index $\leq 8000$.  Note that for some entries in the
table, we found several non-conjugate subgroups with the same data.

\begin{table}[htbp]
  \centering
  \begin{tabular}[htbp]{cccc}
    $H/[H,H]$                        & \# Cusps &  self-intersections   &  $[\Gamma_1:\textrm{Core}_{\Gamma_1}(H)]$ \\
    \hline
    $\Z_2^2 \oplus \Z^4$             &    6     &  $(-2)^6$             & 384\\
    $\Z_2 \oplus \Z_4^2 \oplus \Z^2$ &    4     &  $(-2)^2,(-4)^2$      & 384\\
    $\Z_4^2 \oplus \Z^2$             &    4     &  $(-2)^2,(-4)^2$      & 1536\\
    $\Z_2 \oplus \Z_4^2 \oplus \Z^2$ &    4     &  $(-2)^2,(-4)^2$      & 1536\\
    $\Z_4^2 \oplus \Z^2$             &    4     &  $(-2)^2,(-4)^2$      & 6144\\
  \end{tabular}
  \caption{Numerical invariants for some $TF\infty$ subgroups
    $H\subset \Gamma_1$ of index $96$.}
  \label{tab:subgroups-1}
\end{table}

\subsection{Cusp data for some $TF\infty$ subgroups of index 96 in $\Gamma_2$} \label{sec:subgroups-2}
%%\subsection{The case $d=2$}

In order to get Table~\ref{tab:subgroups-2}, we used normal subgroups
in $\Gamma_2$ of index $\leq 8000$.

\begin{table}[htbp]
  \centering
  \begin{tabular}[htbp]{cccc}
    $H/[H,H]$                          & \# Cusps &  self-intersections       & $[\Gamma_2:\textrm{Core}_{\Gamma_2}(H)]$ \\
    \hline
    $\Z_2^6 \oplus \Z_4 \oplus \Z_8$   &    8     & $(-4)^4,(-8)^4$           & 384\\
    $\Z_2^3 \oplus \Z_4^2$             &    6     & $(-8)^6$                  & 384\\
    $\Z_2^4 \oplus \Z_4^2 \oplus \Z_8$ &    8     & $(-2)^4,(-4)^2,(-16)^2$   & 1536\\
    $\Z_2^6 \oplus \Z_4 \oplus \Z_8$   &    8     & $(-4)^4,(-8)^4$           & 1536\\
  \end{tabular}
  \caption{Numerical invariants for some $TF\infty$ subgroups
    $H\subset \Gamma_2$ of index $96$.}
  \label{tab:subgroups-2}
\end{table}

\subsection{Cusp data for some $TF\infty$ subgroups of index 72 in $\Gamma_3$} \label{sec:subgroups-3}
%\subsection{The case $d=3$}

In order to get Table~\ref{tab:subgroups-3}, we used normal subgroups
in $\Gamma_3$ of index $\leq 7776$ (this was chosen to check whether
we obtain the same subgroups as in~\cite{stover-cusps}, which turns
out to be the case).

\begin{table}[htbp]
  \centering
  \begin{tabular}[htbp]{cccc}
    $H/[H,H]$                          & \# Cusps &  self-intersections      & $[\Gamma_3:\textrm{Core}_{\Gamma_3}(H)]$ \\
    \hline
    $\Z_4$                             &    4     & $(-1)^4$                 & 1944\\
    $\Z_3 \oplus \Z^2$                 &    2     & $(-1),(-3)$              & 1944\\
    $\Z^2$                             &    2     & $(-1),(-3)$              & 5832\\
  \end{tabular}
  \caption{Numerical invariants for some $TF\infty$ subgroups
    $H\subset \Gamma_3$ of index $72$.}
  \label{tab:subgroups-3}
\end{table}

\subsection{Cusp data for some $TF\infty$ subgroups of index 366 in $\Gamma_7$} \label{sec:subgroups-7}
%\subsection{The case $d=7$}

In order to get Table~\ref{tab:subgroups-7}, we used normal subgroups
in $\Gamma_7$ of index $\leq 70000$.

\begin{table}[htbp]
  \centering
  \begin{tabular}[htbp]{cccc}
    $H/[H,H]$                          & \# Cusps &  self-intersections                    & $[\Gamma_7:\textrm{Core}_{\Gamma_7}(H)]$ \\
    \hline
    $\Z_7^8$                           &    24    & $(-7)^{24}$                            & 336\\
    $\Z_2^3\oplus \Z_4^3$              &    12    & $(-2)^2,(-4)^9,(-8)$                 & 10572\\
    $\Z_7\oplus \Z_{14}^3$             &    18    & $(-1)^6,(-2)^3,(-7)^6,(-14)^3$       & 56448\\
    $\Z_2^6\oplus \Z_{6}$              &    12    & $(-1)^3,(-2)^3,(-7)^3,(-14)^3$       & 56448\\
  \end{tabular}
  \caption{Numerical invariants for some $TF\infty$ subgroups
    $H\subset \Gamma_7$ of index $336$.}
  \label{tab:subgroups-7}
\end{table}

\subsection{Cusp data for some $TF\infty$ subgroups of index 432 in $\Gamma_{11}$} \label{sec:subgroups-11}

In order to get Table~\ref{tab:subgroups-11}, we used the neat
principal congruence subgroup of smallest index in $\Gamma_{11}$ (the
corresponding finite quotient has order 5616, and it is isomorphic to
$PSL(3,3)$).

\begin{table}[htbp]
  \centering
  \begin{tabular}[htbp]{cccc}
    $H/[H,H]$                            & \# Cusps &  self-intersections                    & $[\Gamma_{11}:\textrm{Core}_{\Gamma_{11}}(H)]$ \\
    \hline
    $\Z_2\oplus \Z_{12}\oplus \Z_{156}$  &    8     & $(-3)^8$                               & 5616
  \end{tabular}
  \caption{Numerical invariants for some $TF\infty$ subgroups
    $H\subset \Gamma_{11}$ of index $432$.}
  \label{tab:subgroups-11}
\end{table}

\subsection{Cusp data for some $TF\infty$ subgroups of index 864 in $\Gamma_{19}$} \label{sec:subgroups-19}

In order to get Table~\ref{tab:subgroups-19}, we used the neat
principal congruence subgroup of smallest index in $\Gamma_{19}$ (the
corresponding finite quotient has order 6048, and it is isomorphic to
$PU(3,3)$).

\begin{table}[htbp]
  \centering
  \begin{tabular}[htbp]{cccc}
    $H/[H,H]$                          & \# Cusps &  self-intersections                    & $[\Gamma_{19}:\textrm{Core}_{\Gamma_{19}}(H)]$ \\
    \hline
    $\Z_2^5\oplus\Z_6\oplus \Z_{42}$   &    16    & $(-3)^{16}$                            & 6048
  \end{tabular}
  \caption{Numerical invariants for some $TF\infty$ subgroups
    $H\subset \Gamma_{19}$ of index $864$.}
  \label{tab:subgroups-19}
\end{table}

\section{Torsion generating sets} \label{sec:torsion-gens}

Recall that $\Gd$ is generated by torsion elements if and only if
$\Gd\backslash \hc$ is simply connected, by a result of
Armstrong~\cite{armstrong}; in fact, if $\Gd\cong G=\FPG{X}{R}$ and if
we have a list $T$ of representatives of conjugacy classes of torsion
elements in $\Gd$, then $\pi_1(\Gd\backslash \hc)=\FPG{X}{R\cup
  T}$. It is not completely obvious that the latter group can be
computed, but it turns out to be the case for all groups we were able
to treat in this paper.

Specifically, we have the following.
\begin{prop} \label{prop:gen-by-torsion}
  For every $d=1,2,3,7,11,19$, $\Gd$ is generated by torsion
  elements.
\end{prop}

Rather than trying to simplify the presentation $\FPG{X}{R\cup T}$ as
above, we will list explicit torsion elements that generate the
corresponding groups.

In order to obtain such generating sets, we used two different
methods, that turn out to cover all the cases of
Proposition~\ref{prop:gen-by-torsion}.
\begin{enumerate}
\item The first one uses the presentation obtained from the
  Mark-Paupert presentation without simplifying it. Recall that the
  generators of their presentations are obtained from a list of
  rational points $p_k$ by choosing $A_k$ such that
  $A_k(p_\infty)=p_k$. These $A_k$'s are not uniquely defined, but we
  can assume that the set of rational points $\{p_1,\dots,p_N\}$ is
  closed under the matrices $A_k$ (see the discussion in
  section~\ref{sec:feustel-zink}). Once we have adjusted the matrices
  to satisfy this condition, we select only the ones that have finite
  order (this can of course be checked in $\Gd$) to get a list of
  torsion elements $L\subset G=\FPG{X}{R}$ (the word problem in the
  Mark-Paupert generators is easy to solve by geometric means). For
  small subsets $L'=\{X_{j_1},\dots,X_{j_k}\}\subset L$, we would like
  to compute the index of the subgroup generated by $L'$ in $G$ with
  Magma or GAP. When the index is 1, i.e. $L'$ generates $G$, both
  pieces of software answer very quickly; if the index is infinite,
  the computation runs for quite a while, which we take as a sign that
  $L'$ does not seem to generate.

  The easiest way to circumvent this difficulty is to use the Magma
  command
  \begin{center}
    \verb|#Generators(Simplify(G:Preserve:=[j1,...,jk]));|
  \end{center}
  and
  check whether this is equal to $k$ (if so, then $G$ is generated by
  $k$ (torsion) elements).

\item Another method is to use our list of isotropy groups, more
  specifically we take the cyclic groups giving non-reflection
  isotropy group, and use them as candidate torsion generating
  sets. In order to check whether they generate, we write them as
  words in the Mark-Paupert generating set. The command
  \begin{center}
    \verb|Simplify(G:Preserve:=[...]);|
  \end{center}
  does not work directly since these torsion elements are not in the
  Mark-Paupert generating set, but we can add them in the generating
  set by using suitable Tietze transformations (use the Magma command
  \verb|AddGenerator(G,w);|, which creates a new generator and a
  relator that sets it equal to the word $w$).
\end{enumerate}

Method 1 turns out to give 3-generator presentations for $d=2$, $11$
and $19$. For $d=11$ and $19$ we cannot hope for a smaller generating
set, since their abelianizations are given by
$\Gamma_{11}^{(ab)}\cong\Gamma_{19}^{(ab)}\cong \Z_2^3$ (see
Table~\ref{tab:abelianizations} on
page~\pageref{tab:abelianizations}), hence cannot be generated by less
than 3 elements. For $d=2$ we do not know whether there exists a
$2$-generator presentation (note that
$\Gamma_{2}^{(ab)}\cong \Z_2\oplus\Z_4$).

For $d=1,7$, the group $\Gd$ has exactly two conjugacy classes of
isotropy groups that are not generated by complex reflections, and one
checks that method 2 works to show that this gives a 2-element
generating set. For $d=3$, we take the element of order $6$ given by
the complex reflection $R=\textrm{diag}(1,\zeta_6,1)$ and one regular elliptic
isometry of order 4.

\section{Braid presentations}\label{sec:braid-presentations}

For each Picard modular group, we find an explicit small torsion
generating set (see section~\ref{sec:torsion-gens}). Using the list of
(conjugacy classes of) isotropy groups, we get some relations in the
group by
\begin{itemize}
\item expressing the generators of the isotropy groups as words in the
  small torsion generating set;
\item presenting the isotropy groups.
\end{itemize}
Note that the isotropy groups are either reflection groups, or cyclic
groups generated by regular ellitic elements. For non-cyclic
reflection groups, we have explicit presentations (see~\cite{brmaro}
for instance).

We expect that these relations are ``close'' to giving a presentation
of the group, which is confirmed by the results given in
sections~\ref{tab:braid-pres1} through~\ref{tab:braid-pres11}. We omit
the results for $d=19$ because the results barely fit in one page; the
results are more conveniently available in a separate Magma file,
see~\cite{computer-code}.

In order to obtain the presentations below, we
\begin{itemize}
\item start with a version of the Mark-Paupert/Polletta (MPP)
  presentation (simplified so that the generating set is our small
  torsion generating set);
\item add generators corresponding to (minimal) reflection generating
  sets for a representative of each isotropy group, as well as
  generators corresponding to regular elliptic elements for
  non-reflection generators;
\item we include relations corresponding to presentations for the
  corresponding isotropy group;
\item remove as many MPP relations as we can using Magma.
\end{itemize}

We briefly comment on how to perform the last step. The basic point is
that we use the command \verb|SearchForIsomorphism| in Magma.

More specifically, if the small torsion generating set has $k$
elements, we use the command
\begin{center}
  \verb|SearchForIsomorphism(G1,G2,k:MaxRels:=n);|
\end{center}
where $n$ is chosen so that we get an answer in a reasonable amount of
time. Recall that the result of this command is \verb|true| if
Magma finds an isomorphism, and \verb|false| if it did not find one
(which does not necessarily mean that the groups are not
isomomorphic!)

The choice of the Magma parameter $k$ (which is a bound on the sum of
the word lengths of the images of generators) to be equal to the
number of generators is made because we only want to check whether the
obvious map sending the small torsion generating set to themselves is
an isomorphism.

For two (resp. three) generators, $n=10000$ (resp. $n=50000$)
seems to work well in most cases. Note also that removing all MPP at
once seems to represent too much work for Magma, so we remove them
only a few at a time and repeat the procedure.

Note that this is of course not an algorithmic procedure, and the
presentations listed below are by no means canonical. In particular,
even though we hope that the MPP relations that we did not manage to
remove give some information about the global structure of the
orbifold, it is not at all clear how to describe the fundamental group
of the smooth part of the quotient orbifold.

In the presentations of Tables~\ref{tab:braid-pres1}
through~\ref{tab:braid-pres11}, we use the beginning of the alphabet
($a,b$ for $d=1,3,7$, $a,b,c$ for $d=2,11$) for the elements in our
original small torsion generating set. For $d=19$ we only give torsion
generators, but a braid presentation is given at~\cite{computer-code}.

Note that these are chosen to agree with the computer files available
at~\cite{computer-code}, which are generated with the computer; in
many cases there are obvious simplifications (for example, in
Table~\ref{tab:braid-pres1}, the definition of $s$ and $v$ make it
clear that $s=v$, so one could remove one of these from the generating
set).

We then give definitions for generators and relations of isotropy
groups, and finish with (hopefully short) relations that we could not
remove from the MPP relations. We also sometimes keep some relations
that could actually be removed, because they are fairly concise in
writing and give nice group-theoretic/geometric information.

\begin{rmk}
  The relations $kijkjikj$, $ijkijikj$ that appear in the middle of
  the braid relations in Table~\ref{tab:braid-pres2} come from a
  presentation given in~\cite{brmaro} for the Shephard-Todd group
  $G_{12}$ (this occurs as an isoptropy group of the Picard modular
  group $\Gamma_{2}$).
\end{rmk}

For convenience (and perhaps independent interest), we list the
Abelianization of the Picard groups that we were able to compute in
Table~\ref{tab:abelianizations}.
\begin{table}
  \begin{tabular}{|r|c|}
    \hline
    $d$ & $\Gd/[\Gd,\Gd]$\\
    \hline
    $1$ & $\Z/2\Z\oplus \Z/4\Z$\\
    $2$ & $\Z/2\Z\oplus \Z/4\Z$\\
    $3$ & $\Z/6\Z$\\
    $7$ & $\Z/2\Z$\\
    $11$ & $\Z/2\Z\oplus \Z/2\Z\oplus \Z/2\Z$ \\
    $19$ & $\Z/2\Z\oplus \Z/2\Z\oplus \Z/2\Z$ \\
    $43$ & $\Z/2\Z\oplus \Z/2\Z\oplus \Z/2\Z$ \\
    $67$ & $\Z/2\Z\oplus \Z/2\Z\oplus \Z/2\Z$ \\
    $163$ & ? \\
    \hline
  \end{tabular}
  \caption{Abelianizations of $\Gd$}\label{tab:abelianizations}
\end{table}

%\subsection{Braid presentation for $\Gamma_1$ (Abelianization $\Z/2\Z\oplus \Z/4\Z$)}\label{sec:braid-pres1}
%
%Our torsion generating set for $\Gamma_1$ is the set $\{a,b\}$
%described in Table~\ref{tab:torsion-gens1}, and the corresponding
%braid presentation is given in Table~\ref{tab:braid-pres1}.
%\begin{table}[htbp]
%\begin{equation*}
%  a = \left(
%    \begin{matrix}
%      -i  & 0   &-1\\
%      i-1 & -1  &i+1\\
%      i-1 & i-1 &1
%    \end{matrix}
%    \right),\quad
%  b = \left(
%    \begin{matrix}
%      -1  & 0   & i\\
%      i+1 & 1   &-i+1\\
%      1   &-i+1 &-i-1      
%    \end{matrix}
%    \right)
%\end{equation*}
%\caption{Torsion generating set for $\Gamma_1$.}\label{tab:torsion-gens1}
%\end{table}

\begin{table}[htbp]
\centering
\begin{equation*}
  a = \left(
    \begin{matrix}
      -i  & 0   &-1\\
      i-1 & -1  &i+1\\
      i-1 & i-1 &1
    \end{matrix}
    \right),\quad
  b = \left(
    \begin{matrix}
      -1  & 0   & i\\
      i+1 & 1   &-i+1\\
      1   &-i+1 &-i-1      
    \end{matrix}
    \right)
\end{equation*}
\begin{tabular}{|c|c||c|c|}
\hline
  Gens & Isotropy generators & Isotropy relations & Other relations \\
\hline
  $a,b$              & $\begin{array}{c}
                       r = a^{-2}b^{-3}a^2,\\
                       s = aba^{-2}ba^2b^{-1},\\
                       t = a^{-1}b^2a^2b^{-1},\\
                       u = a^2,\\
                       v = aba^{-2}ba^2b^{-1},\\
                       w = aba^{-2}b^{-2}
                        \end{array}$          &  $ \begin{array}{c}
                                                     a^8, b^{12},\\
                                                  r^4, s^2, t^2, u^4, v^2, w^2,\\
                                                  \br_2(r,s),\\
                                                  \br_4(t,u),\\
                                                  \br_3(v,w),\\
                                                \end{array}$    &    $\begin{array}{c}
                                                                        \\
%%                                                                        (a^2b^2)^3,\\
                                                                        \br_2(a^2,ba^2b^{-1})
                                                                      \end{array}$                                           \\
\hline
\end{tabular}
\caption{Torsion generators and braid presentation for $\Gamma_1$}\label{tab:braid-pres1}
\end{table}

%\subsection{Braid presentation for $\Gamma_2$ (Abelianization $\Z/2\Z\oplus \Z/4\Z$)}\label{sec:braid-pres2}
%
%We write $\alpha=i\sqrt{2}$. Our torsion generating set for $\Gamma_2$
%is given by $\{a,b,c\}$ as in Table~\ref{tab:torsion-gens2}, and the
%corresponding braid presentation is given in
%Table~\ref{tab:braid-pres2}.
%\begin{table}[htbp]
%  \begin{equation*}
%    \begin{array}{c}
%      a = \left(
%    \begin{matrix}
%      -2+2\alpha& 2+2\alpha & 3\\
%      2+2\alpha & 3 & 2-2\alpha\\
%      3 & 2-2\alpha & -2-2\alpha
%    \end{matrix}
%    \right),\quad
%     b = \left(
%    \begin{matrix}
%      3+\alpha & 2+\alpha & 1-3\alpha\\
%      -2\alpha & 1-2\alpha & \alpha-4\\
%      -2\alpha & -2\alpha & \alpha-3      
%    \end{matrix}
%    \right),\\
%    c = \left(
%    \begin{matrix}
%      \alpha-2 & 2\alpha-2 & 3\\
%       2\alpha & 2\alpha+1 & 2-\alpha\\
%       3+\alpha& 4+\alpha & -1-3\alpha     
%    \end{matrix}
%                            \right)
%                            \end{array}
%\end{equation*}
%\caption{Torsion generating set for $\Gamma_2$ ($\alpha=i\sqrt{2}$).}\label{tab:torsion-gens2}
%\end{table}

\begin{table}[htbp]
\centering
  \begin{equation*}
    \begin{array}{c}
      a = \left(
    \begin{matrix}
      -2+2\alpha& 2+2\alpha & 3\\
      2+2\alpha & 3 & 2-2\alpha\\
      3 & 2-2\alpha & -2-2\alpha
    \end{matrix}
    \right),\quad
     b = \left(
    \begin{matrix}
      3+\alpha & 2+\alpha & 1-3\alpha\\
      -2\alpha & 1-2\alpha & \alpha-4\\
      -2\alpha & -2\alpha & \alpha-3      
    \end{matrix}
    \right),\\
    c = \left(
    \begin{matrix}
      \alpha-2 & 2\alpha-2 & 3\\
       2\alpha & 2\alpha+1 & 2-\alpha\\
       3+\alpha& 4+\alpha & -1-3\alpha     
    \end{matrix}
                            \right)
                            \end{array}
\end{equation*}
  \begin{tabular}{|c|c||c|c|}
\hline
  Gens & Isotropy generators & Isotropy relations & Other relations \\
\hline
  $a,b,c$              & $\begin{array}{c}
                            e = (cb)^{4},\\
                            f = (bc^{-1})^{2}a(cb)^{2},\\
                            g = (bc^{-1})^{2}a(cb)^{2},\\
                            h = (cb)^{-1}b^{2}cb,\\
                            i = a,\\
                            j = cb^{-1}c^{-1}a^{-1}cb,\\
                            k = cbc^{-1}  a  (cbc^{-1})^{-1},\\
                            l = bc^{-1}a^{-1}cb^{-1},\\
                            m = cba^{-1}cb(cbcbc)^{-1},\\
                            n = bc^{-1}b^{-1}c^{-1}ac,\\
                            o = b^{2},\\
                            p = acbc^{-1}b^{2}c^{-1}a,\\
                            q = b^{-1}c^{-1}acbc^{-1}acb,
                          \end{array}$          &  $ \begin{array}{c}
                                                       a^{2}, b^{4}, c^{6},\\
                                                       e^{2}, f^{2}, g^{2}, h^{2}, i^{2}, j^{2}, k^{2},\\
                                                       l^{2}, m^{2}, n^{2}, o^{2}, p^{8}, q^{4},\\
                                                       \br_2(e,f),\\
                                                       \br_4(g,h),\\
                                                       kijkjikj,\\
                                                       ijkijikj,\\
                                                       \br_3(l,m),\\
                                                       \br_2(n,o),
                                                \end{array}$    &    $\begin{array}{c}
                                                                        (b^{2}cb^{-2}c^{-1})^{2}
                                                                      \end{array}$                                           \\
\hline
  \end{tabular}
  \caption{Torsion generators and braid presentation for $\Gamma_2$ (we write $\alpha=i\sqrt{2}$)}\label{tab:braid-pres2}
\end{table}

%\subsection{Braid presentation for $\Gamma_3$ (Abelianization $\Z/6\Z$)}\label{sec:braid-pres3}
%
%We write $\tau=\frac{1+i\sqrt{3}}{2}$. Our torsion generating set for
%$\Gamma_3$ is given in Table~\ref{tab:torsion-gens3}, and the
%corresponding braid presentation is given in
%Table~\ref{tab:braid-pres3}.
%\begin{table}[htbp]
%  \begin{equation*}
%  a = \left(
%    \begin{matrix}
%      1 & 0                     & 0\\
%      0 & \tau & 0\\
%      0 & 0                     & 1
%    \end{matrix}
%  \right),\quad
%  b = \left(
%    \begin{matrix}
%      -\bar\tau & -1 & \tau+2\\
%      1         & 1  & -2\tau+1\\
%      1         & 0  & -2\tau+1     
%    \end{matrix}
%  \right)
%\end{equation*}
%\caption{Torsion generating set for $\Gamma_3$
%  ($\tau=\frac{1+i\sqrt{3}}{2}$).}\label{tab:torsion-gens3}
%\end{table}

\begin{table}[htbp]
\centering
  \begin{equation*}
  a = \left(
    \begin{matrix}
      1 & 0                     & 0\\
      0 & \tau & 0\\
      0 & 0                     & 1
    \end{matrix}
  \right),\quad
  b = \left(
    \begin{matrix}
      -\bar\tau & -1 & \tau+2\\
      1         & 1  & -2\tau+1\\
      1         & 0  & -2\tau+1     
    \end{matrix}
  \right)
\end{equation*}
  \begin{tabular}{|c|c||c|c|}
\hline
  Gens & Isotropy generators & Isotropy relations & Other relations \\
\hline
  $a,b$              & $\begin{array}{c}
                            e = ba^{-1}b^{-1}a^{-2}b^{2}a^{2}bab^{-1},\\
                            f = ab^{-2}a^{-1},\\
                            g = b^{2}a^{-1}b^{2},\\
                            h = a^{-1}b^{2}a^{2}bab^{-1}
                          \end{array}$          &  $ \begin{array}{c}
                                                       a^{6}, b^4, e^{2}, f^{2}, g^{6}, h^{3},\\
                                                       \br_2(a,e), \br_4(f,g)
                                                     \end{array}$    &  $\begin{array}{c}
                                                                             (ab)^{3}
                                                                           \end{array}$                                           \\
\hline
  \end{tabular}
  \caption{Generators and relations for $\Gamma_3$ (we write $\tau=\frac{1+i\sqrt{3}}{2}$)}\label{tab:braid-pres3}
\end{table}

%\subsection{Braid presentation for $\Gamma_7$ (Abelianization $\Z/2\Z$)}\label{sec:braid-pres7}
%
%In this section, we write $\tau=\frac{1+i\sqrt{7}}{2}$. Our torsion
%generating set for $\Gamma_7$ is given in
%Table~\ref{tab:torsion-gens7}; our braid presentation is given in
%Table~\ref{tab:braid-pres7}.
%\begin{table}[htbp]
%  \begin{equation*}
%  a = \left(
%    \begin{matrix}
%      -2        & \tau-2 & 4\tau-1\\
%      -\bar\tau & \tau+1 & 3\\
%      \tau      & 1      & \bar\tau+2
%    \end{matrix}
%    \right),\quad
%  b = \left(
%    \begin{matrix}
%      -2        & \tau & 3\tau-1\\
%      -\bar\tau & 0    & \tau+2\\
%       \tau     & 1    & \bar\tau+2      
%    \end{matrix}
%    \right)
%\end{equation*}
%\caption{Torsion generating set for $\Gamma_7$
%  ($\tau=\frac{1+i\sqrt{7}}{2}$).}\label{tab:torsion-gens7}
%\end{table}

\begin{table}[htbp]
\centering
  \begin{equation*}
  a = \left(
    \begin{matrix}
      -2        & \tau-2 & 4\tau-1\\
      -\bar\tau & \tau+1 & 3\\
      \tau      & 1      & \bar\tau+2
    \end{matrix}
    \right),\quad
  b = \left(
    \begin{matrix}
      -2        & \tau & 3\tau-1\\
      -\bar\tau & 0    & \tau+2\\
       \tau     & 1    & \bar\tau+2      
    \end{matrix}
    \right)
\end{equation*}
  \begin{tabular}{|c|c||c|c|}
\hline
  Gens & Isotropy generators & Isotropy relations & Other relations \\
\hline
  $a,b$              & $\begin{array}{c}
                          e = b^{-2}a^3b^2,\\
                          f = aba^{-2} \cdot b^{-1} \cdot aba^{-2} \cdot b \cdot aba^{-2},\\
                          g = b^{-1}a^3b,\\
                          h = b^{-1}a^2ba^{3}b,\\
                          i = (ba)^{-1}a^2ba^3(ba),\\
                          j = b (a^{3}ba^2) b^{-1},\\
                          k = b (a^{3}ba^2) b^{-1},\\
                          l = (a^{-1}bab^{-1})a^{-2}ba(a^{-1}bab^{-1})^{-1},\\
                          m = b^{-1}a^2ba^{3}b,\\
                          n = (bab^{-1})^{-1}a^2ba^{3}(bab^{-1})
                          \end{array}$          &  $ \begin{array}{c}
                                                       a^6, b^7,\\
                                                       e^2,f^2,g^2,h^2,i^2,\\
                                                       j^2,k^2,l^2,m^2,n^2,\\
                                                       \br_2(e,f),
                                                       \br_4(g,h),\\
                                                       \br_4(i,j),
                                                       \br_6(k,l),\\
                                                       \br_6(m,n),
                                                     \end{array}$    &  $\begin{array}{c}
                                                                           (bab^{-1}a^{-1}baba^3)^3
                                                                           \end{array}$                                           \\
\hline
  \end{tabular}
  \caption{Torsion generators and braid presentation for $\Gamma_7$ (we write $\tau=\frac{1+i\sqrt{7}}{2}$)}\label{tab:braid-pres7}
\end{table}

%\subsection{Braid presentation for $\Gamma_{11}$ (Abelianization $\Z/2\Z\oplus \Z/2\Z\oplus \Z/2\Z$)}\label{sec:braid-pres11}
%
%In this section we write $\tau=\frac{1+i\sqrt{11}}{2}$. Our torsion
%generating set for $\Gamma_{11}$ is given in
%Table~\ref{tab:torsion-gens11}, and the corresponding braid
%presentation is given in Table~\ref{tab:braid-pres11}.
%\begin{table}[htbp]
%\begin{equation*}
%  a = \left(
%    \begin{matrix}
%      -\tau  & -\tau-1 & \tau-1\\
%      -2     & \tau-2  & 1\\
%      \tau-2 & \tau-1  & 1
%    \end{matrix}
%    \right),\quad
%  b = \left(
%    \begin{matrix}
%      1        & 1     & -\tau\\
%      \bar\tau & -\tau & -2\\
%      -\tau-1  & -1    & \tau-2     
%    \end{matrix}
%    \right),\quad
%  c = \left(
%    \begin{matrix}
%      1 & 0  & 0\\
%      0 & -1 & 0\\
%%      0 & 0  &  1    
%    \end{matrix}
%    \right)
%\end{equation*}
%\caption{Torsion generating set for $\Gamma_{11}$ ($\tau=\frac{1+i\sqrt{11}}{2}$).}\label{tab:torsion-gens11}
%\end{table}

\begin{table}[htbp]
\centering
\begin{equation*}
  a = \left(
    \begin{matrix}
      -\tau  & -\tau-1 & \tau-1\\
      -2     & \tau-2  & 1\\
      \tau-2 & \tau-1  & 1
    \end{matrix}
    \right),\quad
  b = \left(
    \begin{matrix}
      1        & 1     & -\tau\\
      \bar\tau & -\tau & -2\\
      -\tau-1  & -1    & \tau-2     
    \end{matrix}
    \right),\quad
  c = \left(
    \begin{matrix}
      1 & 0  & 0\\
      0 & -1 & 0\\
      0 & 0  &  1    
    \end{matrix}
    \right)
\end{equation*}
  \begin{tabular}{|c|c||c|c|}
\hline
  Gens & Isotropy generators & Isotropy relations & Other relations \\
\hline
  $a,b,c$              & $\begin{array}{c}
                          e = c,\\
                          f = b^{-1}a^{-1}  b^2  ab,\\
                          g = c^{-1}a^{-1}bc^{-1}a^{-1}b\dots\\
\quad\dots c  b^{-1}acb^{-1}ac,\\
                          h = b^{-1}  a^{-2}c^2  b,\\
                          i = c^{-1}a^{-1}bca^{-1}b^{-1}a^{-1}b\dots\\
\quad\dots  c  b^{-1}abac^{-1}b^{-1}ac,\\
                          j = b^{-1}a  b^2  a^{-1}b,\\
                          k = c^{-1}a^{-1}bc^{-1}a^{-2}b^2a^{-1}b,\\
                          l = c^{-1}a^{-1}bca^{-1}b^{-1}a^{-1}\dots\\
\quad\dots b^2a^{-1}ba^{-1}b^{-2}a^{-1}b,\\
                          m = c,\\
                          n = a^{-1}  b^{-1}c^{-2}b^{-1}  a,\\
                          o = b^{-1}a^2b  a^2b^2a^2b^2a^2  b^{-1}a^{-2}b,\\
                          p = b^{-1}a  b^2  a^{-1}b,\\
                          q = a^{-1}  b  a,\\
                          r = b^{-1}  a^{-1}  b,\\
                          s = c^{-1}a^{-1}bc(a^{-1}b^2)^2abcb^{-1}a\dots\\
\quad\dots(b^{-2}a)^2c^{-1}a^{-1}b^{-2}acb^{-1}a^{-2}b
                          \end{array}$          &  $ \begin{array}{c}
                                                       a^4, b^4, c^2,\\
                                                       e^2,f^2,g^2,\\
                                                       h^2,i^2,j^2,\\
                                                       k^2,l^2,m^2,\\
                                                       n^2,o^2,p^2,\\
                                                       q^4,r^4,s^4,\\
                                                       \br_2(e,f),\\
                                                       \br_4(g,h),\\
                                                       \br_2(i,j),\\
                                                       \br_6(k,l),\\
                                                       \br_2(m,n),\\
                                                       \br_3(o,p)
                                                     \end{array}$    &  $\begin{array}{c}
                                                                           (b^{-1}aca^{-1}b^{-1})^2,\\
                                                                           a^{-1}ba^{-1}b^2ab^{-1}ab^{-1}a^2b,\\
                                                                           cb^{-1}ab^{-1}a^{-1}b^{-2}ab^{-1}a^2cb^{-1}a
                                                                         \end{array}$                                           \\
\hline
  \end{tabular}
  \caption{Torsion generators and braid presentation for $\Gamma_{11}$ ($\tau=\frac{1+i\sqrt{11}}{2}$)}\label{tab:braid-pres11}
\end{table}

%\subsection{Braid presentation for $\Gamma_{19}$ (Abelianization $\Z/2\Z\oplus \Z/2\Z\oplus \Z/2\Z$)}\label{sec:braid-pres19}
%
%In this section, we write $u=i\sqrt{19}$.  An explicit torsion
%generating set for $\Gamma_{19}$ is given in
%Table~\ref{tab:torsion-gens19}. We omit writing our braid
%presentation for $\Gamma_{19}$, because it is a bit too long to
%fit on paper, but it is available in the form of a computer
%file~\cite{computer-code}.
%
\begin{table}[htbp]
\begin{equation*}
  \begin{array}{c}
  a = \left(
    \begin{matrix}
     2u+10  & -2u+16 & -7u+10\\
     6  & -2u+3 & -3u-7\\
      -u+2 & -u-5 & -11
    \end{matrix}
    \right),\quad
  b = \left(
    \begin{matrix}
   -\frac{5u+17}{2} & u-14 & \frac{9u-25}{2}\\
     -10  & \frac{5u-7}{2} & \frac{7u+13}{2}\\
     \frac{3u-5}{2}  & u+6 & 11      
    \end{matrix}
    \right),\\
  c = \left(
    \begin{matrix}
     -16  & 6u-10 & \frac{13u+49}{2}\\
      2u-6 & \frac{7u+23}{2} & -u+24\\
      \frac{3u+7}{2} & \frac{-u+25}{2} & \frac{-7u+11}{2}
    \end{matrix}
    \right)\end{array}
\end{equation*}
\caption{Torsion generating set for $\Gamma_{19}$
  ($u=i\sqrt{19}$).}\label{tab:torsion-gens19}
\end{table}


\begin{thebibliography}{10}

\bibitem{alperin}
Roger~C. {Alperin}.
\newblock {An elementary account of Selberg's lemma}.
\newblock {\em {Enseign. Math. (2)}}, 33:269--273, 1987.

\bibitem{armstrong}
M.~A. {Armstrong}.
\newblock {The fundamental group of the orbit space of a discontinuous group}.
\newblock {\em {Proc. Camb. Philos. Soc.}}, 64:299--301, 1968.

\bibitem{amrt}
Avner {Ash}, David {Mumford}, Michael {Rapoport}, and Yung-Sheng {Tai}.
\newblock {\em {Smooth compactifications of locally symmetric varieties}}.
\newblock Cambridge: Cambridge University Press, 2010.

\bibitem{beardon}
Alan~F. {Beardon}.
\newblock {\em {The geometry of discrete groups}}, volume~91 of {\em {Grad.
  Texts Math.}}
\newblock Springer, New York, NY, 1983.

\bibitem{borel}
Armand {Borel}.
\newblock {Introduction aux groupes arithm\'etiques}.
\newblock {Paris: Hermann \& Cie}, 1969.

\bibitem{bhc}
Armand {Borel} and {Harish-Chandra}.
\newblock {Arithmetic subgroups of algebraic groups}.
\newblock {\em {Ann. Math. (2)}}, 75:485--535, 1962.

\bibitem{brmaro}
Michel {Brou\'e}, Gunter {Malle}, and Rapha\"el {Rouquier}.
\newblock {Complex reflection groups, braid groups, Hecke algebras}.
\newblock {\em {J. Reine Angew. Math.}}, 500:127--190, 1998.

\bibitem{deligne-mostow}
P.~{Deligne} and G.~D. {Mostow}.
\newblock {Monodromy of hypergeometric functions and non-lattice integral
  monodromy}.
\newblock {\em {Publ. Math., Inst. Hautes \'Etud. Sci.}}, 63:5--89, 1986.

\bibitem{computer-code}
Martin Deraux.
\newblock gitlab project pic-mod.
\newblock \url{https://plmlab.math.cnrs.fr/deraux/pic-mod}.

\bibitem{deraux-picard}
Martin Deraux.
\newblock On the geometry of a {P}icard modular group.
\newblock Preprint July 2021, \verb|arXiv:2107.09969|.

\bibitem{edmonds-ewing-kulkarni}
Allan~L. {Edmonds}, John~H. {Ewing}, and Ravi~S. {Kulkarni}.
\newblock {Torsion free subgroups of Fuchsian groups and tessellations of
  surfaces}.
\newblock {\em {Invent. Math.}}, 69:331--346, 1982.

\bibitem{falbel-francsics-parker}
Elisha {Falbel}, G\'abor {Francsics}, and John~R. {Parker}.
\newblock {The geometry of the Gauss-Picard modular group}.
\newblock {\em {Math. Ann.}}, 349(2):459--508, 2011.

\bibitem{falbel-parker}
Elisha {Falbel} and John~R. {Parker}.
\newblock {The geometry of the Eisenstein-Picard modular group}.
\newblock {\em {Duke Math. J.}}, 131(2):249--289, 2006.

\bibitem{feustel-holzapfel}
J.-M. {Feustel} and R.-P. {Holzapfel}.
\newblock {Symmetry points and Chern invariants of Picard modular surfaces}.
\newblock {\em {Math. Nachr.}}, 111:7--40, 1983.

\bibitem{feustel}
Jan-Michael {Feustel}.
\newblock {Klassifikation der elliptischen Fixpunkte bez\"uglich der Wirkung
  der Picardschen Modulgruppe auf die komplexe Einheitskugel}.
\newblock {Prepr., Akad. Wiss. DDR, Inst. Math. P-MATH-30/81, 72 S. (1981).},
  1981.

\bibitem{goldman-book}
William~M. {Goldman}.
\newblock {\em {Complex hyperbolic geometry}}.
\newblock Oxford: Clarendon Press, 1999.

\bibitem{holzapfel-book}
Rolf-Peter {Holzapfel}.
\newblock {\em {Ball and surface arithmetics}}, volume E29.
\newblock Wiesbaden: Vieweg, 1998.

\bibitem{feustel-cusps}
Feustel J.-M.
\newblock {\"U}ber die {S}pitzen von {M}odulflächen zur zweidimensionalen
  komplexen {E}inheitskugel.
\newblock Prepr. 13/79, Akad. Wiss. DDR, ZIMM, 1979.

\bibitem{jones-reid}
Kerry~N. {Jones} and Alan~W. {Reid}.
\newblock {Minimal index torsion-free subgroups of Kleinian groups}.
\newblock {\em {Math. Ann.}}, 310(2):235--250, 1998.

\bibitem{kim-parker}
Inkang {Kim} and John~R. {Parker}.
\newblock {Geometry of quaternionic hyperbolic manifolds}.
\newblock {\em {Math. Proc. Camb. Philos. Soc.}}, 135(2):291--320, 2003.

\bibitem{mark-paupert}
Alice Mark and Julien Paupert.
\newblock Presentations for cusped arithmetic hyperbolic lattices.
\newblock \verb|https://arxiv.org/abs/1709.06691|.

\bibitem{mok-projective}
Ngaiming {Mok}.
\newblock {Projective algebraicity of minimal compactifications of
  complex-hyperbolic space forms of finite volume}.
\newblock In {\em Perspectives in analysis, geometry, and topology. On the
  occasion of the 60th birthday of Oleg Viro. Based on the Marcus Wallenberg
  symposium on perspectives in analysis, geometry, and topology, Stockholm,
  Sweden, May 19--25, 2008}, pages 331--354. Basel: Birkh\"auser, 2012.

\bibitem{giac}
B.~Parisse and R.~De Graeve.
\newblock Giac.
\newblock https://www-fourier.ujf-grenoble.fr/$\sim$parisse/giac.html, 2019.

\bibitem{parker-cusps}
John~R. {Parker}.
\newblock {On the volumes of cusped, complex hyperbolic manifolds and
  orbifolds}.
\newblock {\em {Duke Math. J.}}, 94(3):433--464, 1998.

\bibitem{paupert-will}
Julien {Paupert} and Pierre {Will}.
\newblock {Real reflections, commutators, and cross-ratios in complex
  hyperbolic space}.
\newblock {\em {Groups Geom. Dyn.}}, 11(1):311--352, 2017.

\bibitem{polletta}
David {Polletta}.
\newblock {Presentations for the Euclidean Picard modular groups}.
\newblock {\em {Geom. Dedicata}}, 210:1--26, 2021.

\bibitem{rouillier}
F.~Rouillier.
\newblock Solving zero-dimensional systems through the rational univariate
  representation.
\newblock {\em Appl. Algebra Eng. Commun. Comput.}, 9(5):433--461, 1999.

\bibitem{selberg}
Atle {Selberg}.
\newblock {On discontinuous groups in higher-dimensional symmetric spaces}.
\newblock {Contrib. Function Theory, Int. Colloqu. Bombay, Jan. 1960}, 1960.

\bibitem{shephard-todd}
G.~C. Shephard and J.~A. Todd.
\newblock Finite unitary reflection groups.
\newblock {\em Canadian J. Math.}, 6:274--304, 1954.

\bibitem{stover-cusps}
Matthew {Stover}.
\newblock {Cusps of Picard modular surfaces}.
\newblock {\em {Geom. Dedicata}}, 157:239--257, 2012.

\bibitem{zink}
Thomas {Zink}.
\newblock {\"uber die Anzahl der Spitzen einiger arithmetischer Untergruppen
  unit\"arer Gruppen}.
\newblock {\em {Math. Nachr.}}, 89:315--320, 1979.

\end{thebibliography}
\end{document}